\documentclass[10pt]{amsart}
\usepackage{amsmath,amssymb,amsthm,amsfonts,amscd,mathrsfs,mathtools}
\usepackage[utf8]{inputenc}
\usepackage[backend=biber,natbib=true]{biblatex}
\usepackage{todonotes}
\usepackage[hidelinks]{hyperref}

\usepackage[framemethod=TikZ]{mdframed}

\newtheorem{theorem}{Theorem}[section]
\newtheorem{lemma}[theorem]{Lemma}
\newtheorem{conjecture}[theorem]{Conjecture}
\newtheorem{proposition}[theorem]{Proposition}

\newtheorem{corollary}[theorem]{Corollary}
\newtheorem{nota}[theorem]{Notation}
\newtheorem{defn}[theorem]{Definition}
\DeclarePairedDelimiter\floor{\lfloor}{\rfloor}

\DeclareMathOperator{\Q}{\mathbb{Q}}

\newcommand{\fr}{\mathfrak{r}}
\newcommand{\uz}{\underline{0}}

\newcommand{\uc}{\underline{c}}
\newcommand{\uA}{\underline{A}}
\newcommand{\uB}{\underline{B}}

\newcommand{\ual}{\underline{\alpha}}
\newcommand{\ube}{\underline{\beta}}
\newcommand{\uep}{\underline{\epsilon}}

\newcommand{\A}{\mathcal{A}}
\newcommand{\tA}{\tilde\A}
\newcommand{\bB}{\mathbb{B}}
\newcommand{\C}{\mathbb{C}}
\newcommand{\fg}{\mathfrak{g}}
\newcommand{\tfg}{\tilde{\fg}} 
\newcommand{\fS}{\mathfrak{S}}
\newcommand{\N}{\mathbb{N}}
\newcommand{\Z}{\mathbb{Z}}
\DeclareMathOperator{\grad}{grad}
\usepackage{mathbbol}
\DeclareSymbolFontAlphabet{\mathbb}{AMSb}
\DeclareSymbolFontAlphabet{\mathbbl}{bbold}
\newcommand{\one}{\mathbbl{1}} 
\newcommand{\supp}{\mathrm{supp}} 
\newcommand{\tr}{\mathbf{t}} 
\renewcommand{\l}{\lambda} 
\newcommand{\ul}{\underline{\l}} 
\newcommand{\Hilb}{\mathrm{Hilb}} 
\newcommand{\geo}{^{\mathrm{geo}}} 
\newcommand{\QE}{\Q\! E} 

\newcommand{\del}{\partial}
\newcommand{\set}[1]{\left\{#1\right\}}
\newcommand{\pa}[1]{\left(#1\right)}

\title[Generators and relations of cohomology of Hilbert schemes]{On generators and relations of the rational cohomology of Hilbert schemes}
\author{Andrea Bianchi}
 \thanks{Andrea Bianchi was supported by the Danish National Research Foundation through the Centre for Geometry and Topology (DNRF151) and by the European Research Council under the European Union’s Horizon 2020 research and innovation programme (grant agreement No. 772960)}
\email{anbi@math.ku.dk}
\address{Department of Mathematical Sciences, University of Copenhagen \newline
Universitetsparken 5, Copenhagen, 2100, Denmark}  

\author[Alexander M. Christgau]{Alexander Mangulad Christgau}
\thanks{}
\email{amc@math.ku.dk}
\address{Department of Mathematical Sciences, University of Copenhagen \newline
Universitetsparken 5, Copenhagen, 2100, Denmark}  

\author[Jonathan S. Pedersen]{Jonathan Sejr Pedersen$^\ast$}
\thanks{$^\ast$Corresponding author.}
\email{jsejrp@math.toronto.edu}
\address{Department of Mathematics,
University of Toronto \newline
Bahen Centre
40 St. George St., M5S 2E4,
Toronto, Ontario,
Canada}

\date{\today}

\begin{document}
\begin{abstract}
We consider for $d\ge1$ the graded commutative $\mathbb{Q}$-algebra $\mathcal{A}(d):=H^*(\operatorname{Hilb}^d(\mathbb{C}^2);\mathbb{Q})$, which is also connected to the study of generalised Hurwitz spaces by work of the first author. These Hurwitz spaces are in turn related to the moduli spaces of Riemann surfaces with boundary.

We determine two distinct, minimal sets of $\floor{d/2}$ multiplicative generators of $\mathcal{A}(d)$. Additionally, we prove when the lowest degree generating relations occur. For small values of $d$ we also determine a minimal set of generating relations, which leads to several conjectures about the necessary generating relations for $\mathcal{A}(d)$.
\end{abstract}
\maketitle
\section{Introduction}
The main object of study in this article is the algebra $\A(d):=H^*(\Hilb^d(\C^2);\Q)$, i.e. the \emph{rational} cohomology algebra of the $d$\textsuperscript{th} Hilbert scheme of $\C^2$.
It is a finite dimensional, graded commutative $\Q$-algebra concentrated in even degrees \cite{Ellingsrud1993}; it admits a combinatorial description
as the associated graded of a certain \emph{norm filtration} on the rational group algebra $\Q[\fS_d]$ of the symmetric group $\fS_d$ on $d$ letters \cite{LehnSorger}. Thanks to this combinatorial description, $\A(d)$ can be identified as a commutative $\Q$-algebra with
$\Q[\fS_d^{\mathrm{geo}}]^{\fS_d}$, i.e. the conjugation invariants of the rational PMQ-algebra $\Q[\fS_d^{\mathrm{geo}}]$ associated with the \emph{partially multiplicative quandle} $\fS_d^{\mathrm{geo}}$ from \cite[Section 7]{Bianchi:Hur1}.

We define the \emph{norm} of a homogeneous element of even degree in a graded $\Q$-vector space to be half of the degree.
We denote by $\A(d)_{j}$ the homogeneous component of $\A(d)$ of norm $j$, and by $\A(d)_+=\bigoplus_{j\ge1}\A(d)_{j}$ the augmentation ideal, spanned by all elements of positive norm.

\subsection{Statement of results}
We study the numbers of generators and relations necessary to present $\A(d)$ as a graded commutative $\Q$-algebra.
\begin{theorem}
\label{thm:minimalgenerators}
A minimal set of homogeneous generators for $\A(d)$ as a $\Q$-algebra contains one element in each norm $j$ with $1\le j\le d/2$.

In other words, the module of \emph{indecomposables}
$\A(d)_+/\A(d)_+^2$
has dimension 1 in each norm $1\le j\le d/2$, and vanishes elsewhere.
\end{theorem}
We fix a generating set $x_1,\dots,x_m$ for $\A(d)$, where $m\colon =\lfloor d/2\rfloor$ and $x_j$ has degree $j$. An explicit choice of such generators will be given in Section \ref{sec:minimalgenerators}. We consider the graded polynomial ring $\Q[X_1,\dots,X_m]$, where $X_j$ has degree $2j$, i.e. norm $j$, and we denote by $\Q[X_1,\dots,X_m]_+$ the augmentation ideal, spanned by monomials of positive norm. By Theorem \ref{thm:minimalgenerators} the graded algebra homomorphism $\phi\colon\Q[X_1,\dots,X_m]\to\A(d)$, given by $X_j\mapsto x_j$, is surjective; the homogeneous ideal $I_d:=\mathrm{ker}(\phi)$ contains all relations among the generators $x_1,\dots,x_m$; a set of \emph{minimal relations} for $\A(d)$ is a set of homogeneous elements of $I_d$ projecting to a vector space basis of the quotient $I_d/\Q[X_1,\dots,X_m]_+\cdot I_d$; the latter is by definition the module of \emph{indecomposable relations} of $\A(d)$, and it is a graded $\Q$-vector space.
\begin{theorem}
\label{thm:minimalrelationsbottom}
The generators $x_1,\dots,x_m$ are algebraically independent in norms up to $d-m$. For $d=2m\ge2$ even, there is precisely one minimal relation among $x_1,\dots,x_m$ in norm $m+1$; for $d=2m+1\ge5$ odd there are precisely two minimal relations among $x_1,\dots,x_m$ in norm $m+2$.

In other words, the module of indecomposable relations of $\A(d)$ is concentrated in norms at least $d-m+1$, and it has dimension 1 (for $d\ge2$ even) or dimension 2 (for $d\ge5$ odd) in norm $d-m+1$.
\end{theorem}
We remark that $\A(1)\cong\Q$ and $\A(3)\cong\Q[X]/X^3$, with $X$ in degree 2.
We provide in fact explicit minimal presentations of $\A(d)$ for $d\le 10$, i.e. we compute explicitly the module of indecomposable relations: see Section \ref{sec:smallcases} and Appendix \ref{app:A9A10}.

\subsection{Motivation}
Our interest for minimal generators and relations of $\A(d)$ comes from the first author's work on generalised Hurwitz spaces \cite{Bianchi:Hur1,Bianchi:Hur2,Bianchi:Hur3,Bianchi:Hur4}. The algebra $\A(d)$ is also isomorphic to the rational cohomology algebra of a certain simply connected space $\bB(\fS_d\geo,\fS_d)$; one is however rather interested in computing the rational cohomology ring $H^*(\Omega^2_0\bB(\fS_d\geo,\fS_d);\Q)$ of a component of the double loop space, because this ring can be identified with the ring of stable rational cohomology classes of Hurwitz spaces with monodromies in $\fS_d\geo$. These Hurwitz spaces are in turn strongly related to moduli spaces of Riemann surfaces with boundary.

The natural approach to compute $H^*(\Omega^2_0\bB(\fS_d\geo,\fS_d);\Q)$ is to first compute a minimal Sullivan model for $\bB(\fS_d\geo,\fS_d)$, and then to desuspend twice this model. Among the generators of a minimal Sullivan model, we must find some that account for a minimal generating set of the cohomology ring  $\A(d)$, and some that account for a minimal set of relations among the minimal generators of $\A(d)$. Hence our computations give a lower bound on the number of generators of a minimal Sullivan model of $\bB(\fS_d\geo,\fS_d)$.

\subsection{Acknowledgments}
We would like to thank S\o ren Eilers for putting the first author in contact with the second and third author through the course ``Experimental mathematics'' at the University of Copenhagen. The first author would like to thank S\o ren Galatius for a helpful conversation on Sullivan models and $\QE_\infty$-homology. We thank the referees for their useful comments that helped improve the manuscript.

\tableofcontents

\section{Preliminaries on \texorpdfstring{$\A(d)$}{A(d)}}
\subsection{Symmetric groups as normed groups}
\label{subsec:defnAd}
Let $d\geq 2$, and consider the symmetric group $\fS_d$. It is a \emph{normed group}, i.e. there
exists a norm function $N\colon \fS_d\to\N$
with the following properties:
\begin{itemize}
 \item $N(\sigma)=0$ if and only if $\sigma=\one$;
 \item $N(\sigma)=N(\tau\sigma\tau^{-1})$ for all $\sigma,\tau$;
 \item $N(\sigma\tau)\leq N(\sigma)+N(\tau)$.
\end{itemize}
Explicitly, for all $\sigma\in\fS_d$, we define $N(\sigma)$ as the minimum $k\geq0$ such that
$\sigma$ can be written as a product $\tr_1\cdots\dots\cdot\tr_k$ of transpositions.
An equivalent characterisation of the norm is the following: if $\sigma\in\fS_d$ has a cycle decomposition with $1\le k\le d$ cycles, where fixpoints count as cycles of length 1, then $N(\sigma)=d-k$; see for instance \cite[Lemma 7.3]{Bianchi:Hur1}.
The following is \cite[Lemma 7.4]{Bianchi:Hur1}.
\begin{lemma}
 \label{lem:geodesictransposition}
 Let $\sigma\in\fS_d$ be a permutation and $\tr=(ij)\in\fS_d$ be a transposition interchanging the elements $i,j\in\set{1,\dots,d}$. Then $N(\tr\sigma)=N(\sigma)+1$
 if $i$ and $j$ belong to different cycles of the cycle decomposition of $\sigma$, and $N(\tr\sigma)=N(\sigma)-1$
 otherwise.
\end{lemma}

The same result holds in Lemma \ref{lem:geodesictransposition} if we consider the product $\sigma\tr$ instead of $\tr\sigma$. The following is \cite[Corollary 7.5]{Bianchi:Hur1}
\begin{corollary}
 \label{cor:geodesicpair}
 Let $\sigma,\tau\in\fS_d$ and suppose that $N(\sigma\tau)=N(\sigma)+N(\tau)$. Let $c$ be a cycle in
 the cycle decomposition of $\sigma$, and consider $c$ as a subset of $\{1,\dots,d\}$. Then
 $c$ is contained in some cycle $c'$ of the cycle decomposition of $\sigma\tau$.
\end{corollary}

\subsection{From \texorpdfstring{$\Q[\fS_d]$}{Q[Sd]} to \texorpdfstring{$\A(d)$}{A(d)}}
The group algebra $\Q[\fS_d]$ is then a filtered $\Q$-algebra: for all $k\geq0$ we can define $F^N_k=F_k\Q[\fS_d]$ as the sub-vector space spanned by all $\sigma$ with $N(\sigma)\leq k$; we also set by convention $F^N_{-1}=0$. We have $F^N_k\cdot F^N_{k'}\subseteq F^N_{k+k'}$ for all $k,k'\geq0$, as the norm satisfies the triangle inequality.
Hence the associated graded 
\[
\grad^N\Q[\fS_d]:=\bigoplus_{k\geq0} F^N_k/F^N_{k-1}
\]
is a graded $\Q$-algebra. Explicitly, $\grad^N\Q[\fS_d]$
is generated as a vector space by elements $[\sigma]$ of norm $N(\sigma)$, for all $\sigma\in\fS_d$; we regard $[\sigma]$ as an element of degree $2N(\sigma)$ for later convenience. The multiplication is given by the following formula:
\[
 [\sigma]\cdot[\tau]=\left\{
 \begin{array}{cl}
  [\sigma\tau] & \mbox{if } N(\sigma\tau)=N(\sigma)+N(\tau);\\
  0 & \mbox{if }N(\sigma\tau)<N(\sigma)+N(\tau).
 \end{array}
 \right.
\]
Note that this algebra is just $\Q\cong\Q\left<\one\right>$ in norm $0$, since the only norm $0$ permutation is the identity permutation $\one\in\fS_d$.

The symmetric group $\fS_d$ acts on this entire construction by conjugation: it acts by $\Q$-algebra automorphisms
on $\Q[\fS_d]$, preserving the filtration $F^N_{\bullet}$: here we use that the norm is conjugation invariant.
Hence there is an induced action of $\fS_d$ by automorphisms of graded algebra on $\grad^N\Q[\fS_d]$. By the main result of \cite{LehnSorger} we have an isomorphism of graded $\Q$-algebras
\[
 \A(d):=H^*(\Hilb^d(\C^2);\Q)\cong\pa{\grad^N\Q[\fS_d]}^{\fS_d},
\]
where the right hand side is the subalgebra of $\fS_d$-invariants. In particular the right hand side is a commutative algebra: this follows also from \cite[Lemma 4.31]{Bianchi:Hur1}.

From now on we will consider $\A(d)$ as a subalgebra of 
$\grad^N\Q[\fS_d]$.
Before proceeding in our study of $\A(d)$, we have to fix some notation that will remain in effect throughout the paper.

\subsection{Some notation}\label{subsec:notation}
Recall that a conjugacy class in $\fS_d$ is determined by its cycle type, which is to say, a sequence $\ul=(\l_2,\l_3,\dots)$ of
integers $\geq0$, with almost all terms equal to 0, such that number $\supp(\ul):=\sum_{i=2}^\infty i\l_i$
is $\leq d$. For such a sequence
define $(\ul)\subset\fS_d$ to be the corresponding conjugacy class, containing all permutations whose cycle
decomposition has precisely $\l_i$ cycles of length $i$ for all $i\geq2$, and the right number of fixpoints to reach $d$: note that the \emph{support} $\supp(\sigma)$ of such a permutation, i.e. the cardinality of the subset
of $\set{1,\dots,d}$ on which $\sigma$ acts non-trivially, is equal to $\supp(\ul)$. Define
\[
 \fg_{\ul}=\fg_{\ul}(d):=\sum_{\sigma\in(\ul)}[\sigma].
\]
Then the elements $\fg_{\ul}$ belong to $\A(d)\subset\grad^N\Q[\fS_d]$, and
for $(\ul)$ ranging among conjugacy classes of $\fS_d$, these elements form a basis of $\A(d)$ as
$\Q$-vector space.

For $j\ge 0$ we denote by $\A(d)_{j}$ the norm $j$ summand of $\A(d)$,
and by $\A(d)_{\le j}$ the direct sum of all $\A(d)_{j'}$ for $j'\le j$.
For $k\ge0$ we denote $\Lambda(k,d)$ the set of all sequences $\ul=(\l_2,\l_3,\dots)$ as above with $\supp(\ul)\le d$ and such that the number $N(\ul):=\sum_{i=2}^{\infty}(i-1)\l_i$ is equal to $k$. Then $\A(d)_{j}$ has a basis given by the elements $\fg_{\ul}$ for $\ul$ in $\Lambda(j,d)$. The following two families of basis elements will play a central role as each of them turns out to give a generating set for $\A(d)$, as a $\Q$-algebra:
\begin{itemize}
    \item For $1\le i\le m:=\lfloor d/2\rfloor$ we denote by $\delta_i$ the element $\fg_{\ul}\in\A(d)_{i}$ corresponding to the sequence $\ul=(\l_2,\l_3,\dots)$ with $\l_2=i$ and all $\l_j=0$ for $j\neq 2$.
    \item For $2\le i\le d$ we denote by $\one_i$ the sequence $\ul=(\l_2,\l_3,\dots)$ with $\l_i=1$ and all $\l_j=0$ for $j\neq i$, and we denote by $\gamma_i$ the element $\fg_{\one_i}\in\A(d)_{i-1}$.
\end{itemize}
Whenever there is risk for confusion, i.e. there is an argument involving two algebras $\A(d)$ and $\A(h)$
for $d\le h$, we will write $\fg_{\ul}(d)$ for a basis element of $\A(d)$, and $\fg_{\ul}(h)$ for a basis
element of $\A(h)$: if the sequence $\ul=(\l_2,\l_3,\dots)$ satisfies $\supp(\ul)\leq d\le h$, then
$\ul$ gives indeed rise to basis elements of both $\A(d)$ and $\A(h)$.

\section{A sufficient set of generators}
Recall that for $2\le i\le d$ we have the basis elements $\gamma_i=\fg_{\one_i}$ of $\A(d)$ corresponding to the sequences $\one_i$. In this section we prove the following proposition.
\begin{proposition}
\label{prop:sufficientgenerators}
The elements $\gamma_2,\dots,\gamma_{d}$ generate $\A(d)$ as a $\Q$-algebra.
\end{proposition}

It will be crucial for our argument that we are working over the field
$\Q$ and not over a field of finite characteristic.

\subsection{The support filtration}
Recall that for $\sigma\in\fS_d$ we have defined the \emph{support}
\[
 \supp(\sigma)=\sum_{i=2}^{\infty}i\l_i,
\]
where $\l_i$ is the number of $i$-cycles in the cycle decomposition of $\sigma$. The support measures how many
non-fix-points the permutation $\sigma$ has, and it is an invariant of the conjugacy class of $\sigma$.
Note also that $\supp(\sigma)=0$ if and only if $\sigma=\one$, and that the triangular inequality $\supp(\sigma\tau)\leq \supp(\sigma)+\supp(\tau)$ holds. Therefore, the support gives an alternative norm
on the group $\fS_d$. We thus have a new conjugation invariant, multiplicative filtration
$F^{\supp}_{\bullet}$ on $\Q[\fS_d]$, inducing a multiplicative filtration, also called $F^{\supp}_{\bullet}$ for simplicity, on $\grad^N\Q[\fS_d]$. Explicitly, for $k\ge0$, the vector space
$F^{\supp}_k\subset\grad^N\Q[\fS_d]$ is spanned by the basis elements $[\sigma]$ corresponding to all permutations $\sigma\in\fS_d$ with $\supp(\sigma)\leq k$.

We take the associated graded $\grad^{\supp}\grad^N\Q[\fS_d]$,
which is a bigraded algebra, generated as a vector space by elements $[[\sigma]]$ for all permutations
$\sigma\in\fS_d$, with multiplication rule
\[
 [[\sigma]]\cdot[[\tau]]\!=\!\left\{\!\!
 \begin{array}{cl}
  [[\sigma\tau]] & \mbox{if } N(\sigma\tau)=N(\sigma)+N(\tau)\mbox{ and }\supp(\sigma\tau)=\supp(\sigma)+\supp(\tau);\\
  0 & \mbox{otherwise}.
 \end{array}
 \right.
\]
The element $[[\sigma]]$ lives in bigrading $(N(\sigma),\supp(\sigma))$.
Note that if $\sigma$ and $\tau$ are permutations such that $\supp(\sigma\tau)=\supp(\sigma)+\supp(\tau)$,
then the set of non-fixed-points of $\sigma$ must be disjoint from the set of non-fixed points of $\tau$ (both are subsets of $\set{1,\dots,d}$); in this case, the further equality $N(\sigma\tau)=N(\sigma)+N(\tau)$ follows automatically.

We can now take $\fS_d$-invariants of $\grad^{\supp}\grad^N\Q[\fS_d]$, or equivalently restrict
the filtration $F^{\supp}_{\bullet}$ from $\grad^N\Q[\fS_d]$ to the subalgebra $\A(d)$, and then take the associated graded: in both cases
we obtain a bigraded $\Q$-algebra $\tA(d)$, with basis as a $\Q$-vector space given by the elements
\[
 \tfg_{\ul}=\sum_{\sigma\in(\ul)}[[\sigma]],
\]
for $\ul$ ranging in $\Lambda(0,d)\cup\dots\dots\Lambda(d-1,d)$.

Since $\A(d)$ was commutative, also $\tA(d)$ is commutative. We want now to check that $\tA(d)$ is generated
as a $\Q$-algebra by the elements $\tilde{\gamma}_i$ for $2\leq i\leq d$: this suffices to conclude that also
$\A(d)$ is generated by the elements $\gamma_i$ for $2\leq i\leq d$, thanks to the following well-known statement.
\begin{lemma}
\label{lem:associatedgraded}
Let $R$ be a $\Q$-algebra with a multiplicative filtration $F_{\bullet}$, with $F_{-1}=0$ and $F_d=R$ for some $d\ge0$. Let $x_1,\dots,x_r$ be elements of $R$, with $x_j\in F_{k_j}$, for suitably chosen integers $k_j$.
Denote by $\bar x_j$ the image of $x_j$ in $F_{k_j}/F_{k_j-1}\subset\grad R$. Assume that the elements
$\bar x_1,\dots,\bar x_r$ generate $\grad R$ as a $\Q$-algebra.
Then the elements $x_1,\dots,x_r$ generate $R$ as a $\Q$-algebra.
\end{lemma}

\subsection{Generating all basis elements}
Let $\ul=(\l_2,\l_3,\dots)$ give a basis element $\tfg_{\ul}$ of $\tA(d)$. Consider the product
\[
 \pi=(\tilde{\gamma}_2)^{\l_2}\cdot(\tilde{\gamma}_3)^{\l_3}\cdot\dots(\tilde{\gamma}_d)^{\l_d}\in\tA(d).
\]
Then it suffices to prove that $\pi$ is a non-zero multiple of $\tfg_{\ul}$, to show that $\tfg_{\ul}$ can be generated multiplicatively by our generators. We have indeed
\[
\pi=(\l_2)!\cdot(\l_3)!\dots(\l_d)!\cdot\tfg_{\ul}\in\tA(d).
\]
This formula follows immediately from the following multiplication formula for $\tA(d)$: for $\ul,\ul'$ with $\supp(\ul),\supp(\ul')\le d$ we have
\[
 \tfg_{\ul}\cdot\tfg_{\ul'}=\left\{
 \begin{array}{cl}
  \binom{\l_2+\l'_2}{\l_2}\cdot\binom{\l_3+\l'_3}{\l_3}\dots\binom{\l_d+\l'_d}{\l_d}\cdot\tfg_{\ul+\ul'},
  &\mbox{if }\supp(\ul)+\supp(\ul')\leq d\\[9pt]
  0 &\mbox{otherwise}.
 \end{array}
 \right.
\]
Here $\ul=(\l_2,\l_3,\dots)$ and $\ul'=(\l'_2,\l'_3,\dots)$ respectively, and we denote $\ul+\ul':=(\l_2+\l'_2,\l_3+\l'_3,\dots)$.
The multiplication formula is justified as follows.
\begin{itemize}
 \item If we compute $\tfg_{\ul}\cdot\tfg_{\ul}$ in the larger ring $\grad^{\supp}\grad^N\Q[\fS_d]$, we have to sum all products $[[\sigma]]\cdot[[\tau]]$ for all $\sigma\in(\ul)$
 and $\tau\in(\ul')$; the product $[[\sigma]]\cdot[[\tau]]$ vanishes unless $\sigma$ and $\tau$ have disjoint supports (as subsets of  $\set{1,\dots,d}$), in which
 case $[[\sigma]]\cdot[[\tau]]=[[\sigma\tau]]$, where $\sigma\tau$ is a permutation with $\l_i+\l'_i$ cycles of length
 $i$, for all $i\ge 2$; hence $\sigma\tau\in(\ul+\ul')$.
 \item Given a generic permutation $\rho\in(\ul+\ul')$, there are
 $\prod_{i\ge2}\binom{\l_i+\l'_i}{\l_i}$ ways to write it
 as a product $\sigma\tau$, with $\sigma\in(\ul)$ and $\tau\in (\ul')$; note that it is automatic that
 $\sigma$ and $\tau$ have disjoint supports, hence for each such decomposition of $\rho$ we have indeed
 $[[\sigma]]\cdot[[\tau]]=[[\rho]]$.
\end{itemize}
The previous argument shows that the elements $\tilde\gamma_i$, for $i\ge2$, generate $\tA(d)$ as a $\Q$-algebra. By Lemma \ref{lem:associatedgraded} we obtain that the elements $\gamma_i$, for $i\ge2$, generate $\A(d)$ as a $\Q$-algebra, which is the statement of Proposition \ref{prop:sufficientgenerators}.

\subsection{Algebraic independence in low norm}
\label{subsec:algebraicindependence}
The next step is to show that the generators $\gamma_2,\dots,\gamma_d$ of $\A(d)$ have no relations in small norm. Let $X_i$ be a variable in degree $2i$, i.e. in norm $i$, and let $\Q[X_1,\dots,X_{d-1}]$ denote the polynomial
algebra in the variables $X_i$, which is also a graded commutative algebra. By sending $X_i\mapsto \gamma_{i+1}$ we obtain a surjective map of graded $\Q$-algebras
\[
 \phi\colon \Q[X_1,\dots,X_{d-1}]\to\A(d).
 \]
\begin{proposition}
\label{prop:algebraicindependence}
For all $d\ge1$ and all $j\leq d/2$, the map $\phi$ restricts to an isomorphism of vector spaces $\Q[X_1,\dots,X_{d-1}]_{j}\cong\A(d)_{j}$.
\end{proposition}
\begin{proof}
We already know that $\phi$ restricts to a surjective map
$\Q[X_1,\dots,X_{d-1}]_{j}\to\A(d)_{j}$ for all $j\le d/2$; so it suffices to
check that, for $j\leq d/2$, the $\Q$-vector spaces $\Q[X_1,\dots,X_{d-1}]_{j}$ and $\A(d)_{j}$ have the same dimension.

The dimension of $\Q[X_1,\dots,X_{d-1}]_{j}$ is equal to the number of sequences $\ul=(\l_2,\l_3,\dots)$ satisfying the following condition
\begin{itemize}
 \item $N(\ul)=\sum_{i=2}^\infty (i-1)\l_i=j$.
\end{itemize}
Indeed, for each such $\ul$ we can form the monomial $X_1^{\l_2}\cdot\dots X_{d-1}^{\l_d}$, and these monomials
form a basis of the norm $j$ component of the polynomial algebra $\Q[X_1,\dots,X_{d-1}]$: here we use in particular the inequality $j\le d-1$.

On the other hand, the dimension of $\A(d)_{j}$ is equal to the cardinality of $\Lambda(j,d)$, i.e. the number of sequences $\ul=(\l_2,\l_3,\dots)$ satisfying the following conditions
\begin{itemize}
 \item $N(\ul)=\sum_{i=2}^\infty (i-1)\l_i=j$;\vspace{4pt}
 \item $\supp(\ul)=\sum_{i=2}^\infty i\l_i\leq d$.
\end{itemize}
We note now that for $j\leq d/2$ the second condition is implied by the first, since 
\[
 \sum_{i=2}^\infty i\l_i\leq \sum_{i=2}^\infty (2i-2)\l_i=2j\leq d.
\]
This concludes the proof of Proposition \ref{prop:algebraicindependence} (see also \cite[Remark 6.3]{LehnSorger}).
\end{proof}

A consequence of Proposition \ref{prop:algebraicindependence} is that any generating set for $\A(d)$ as a $\Q$-algebra must contain at least one element in each norm $1\le j\le d/2$; more precisely, Propositions \ref{prop:sufficientgenerators} and \ref{prop:algebraicindependence} together imply that the module of indecomposables $\A(d)_+/\A(d)_+^2$ has dimension 1 in each norm $1\le j\le d/2$, i.e. in each even degree $2\le 2j\le d$; of course, there are no indecomposables in odd degrees.

\subsection{A useful map of algebras}
We conclude the section by constructing, for $h\geq d\geq2$, a map of algebras $p^h_d\colon \A(h)\to\A(d)$. In this subsection
we use the expanded notation for the generators to avoid ambiguity:
given a sequence $\ul=(\l_2,\l_3,\dots)$ with $\supp(\ul)\leq d$,
we denote by $\fg_{\ul}(d)$ and $\fg_{\ul}(h)$ the corresponding basis elements in $\A(d)$ and $\A(h)$.
Hence, for instance,
\[
 \fg_{\ul}(d)=\sum_{\sigma\in(\ul)\subset\fS_d}[\sigma]\in\A(d).
\]
It is also convenient to introduce the notation $\fg_{\ul}(d)\in\A(d)$ for \emph{all} sequences
$\ul=(\l_2,\l_3,\dots)$, not only those satisfying $\supp(\ul)\leq d$. So we set
\[
 \fg_{\ul}(d):=0
\]
whenever $\supp(\ul)> d$. Note that, for any sequence $\ul$, if $\fg_{\ul}(h)=0$, then
$\supp(\ul)>h$, and therefore we also have $\supp(\ul)>d$, which implies
$\fg_{\ul}(d)=0$.

We define the map $p^h_d$ as a map of $\Q$-vector spaces, by sending the basis element
$\fg_{\ul}(h)$ to the element $\fg_{\ul}(d)$; note that $\fg_{\ul}(d)$ is always either zero or
a basis element of $\A(d)$.

We need to prove that $p^h_d$ is compatible with the multiplication.
For this let us introduce some notation.

 For all $k_1,k_2\ge0$ and for all $\ual=(\alpha_2,\alpha_3,\dots)\in\Lambda(k_1,d)$ and $\ube=(\beta_1,\beta_2,\dots)\in\Lambda(k_2,d)$, the product $\fg_{\ual}(d)\cdot\fg_{\ube}(d)\in\A(d)$ is a linear combination of the basis elements $\fg_{\uep}(d)$, for $\uep=(\epsilon_2,\epsilon_3,\dots)$ ranging in $\Lambda(k_1+k_2,d)$.
 
\begin{nota} 
\label{nota:productcoefficient}
For $\ual,\ube,\uep$
satisfying $\max(\supp(\ual),\supp(\ube),\supp(\uep))\le d$ and $N(\ual)+N(\ube)=N(\uep)$
we define the coefficient
$\theta_d(\uep;\ual,\ube)\in\Z_{\ge0}$ 
so that the following formula holds
\[
  \fg_{\ual}\fg_{\ube}=\sum_{\uep}\theta_d(\uep;\ual,\ube)\,\fg_{\uep}.
\]
\end{nota}
The coefficient
$\theta_d(\uep;\ual,\ube)$
has a combinatorial interpretation: given any permutation
$\rho\in(\uep)\subset\fS_d$, we have that $\theta_d(\uep;\ual,\ube)$ is the number of factorisations
$\rho=\sigma\cdot\tau$ with $\sigma\in(\ual)\subset\fS_d$ and $\tau\in(\ube)\subset\fS_d$. This follows
from the following remarks.
\begin{itemize}
 \item If we multiply $\fg_{\ual}(d)\cdot\fg_{\ube}(d)$ inside $\grad^N\Q[\fS_d]$, we will obtain
 the term $[\rho]$ as many times as there are factorisations $\rho=\sigma\tau$ as above.
 \item Conversely, if $\sigma\in(\ual)\subset\fS_d$ and $\tau\in(\ube)\subset\fS_d$, then the product
 $[\sigma]\cdot[\tau]$, which is one of the summands when computing $\fg_{\ual}(d)\cdot\fg_{\ube}(d)$,
 will either be zero, or equal to $[\rho]=[\sigma\tau]$ with $\rho=\sigma\tau\in\fS_d$ being a permutation of norm
 $N(\ual)+N(\ube)$.
\end{itemize}
\begin{lemma}
\label{lem:phd}
Let $h\ge d$, let $k_1,k_2\ge0$, and let $\uep\in\Lambda(k_1+k_2,d)\subset\Lambda(k_1+k_2,h)$, $\ual\in\Lambda(k_1,d)\subset\Lambda(k_1,h)$ and $\ube\in\Lambda(k_2,d)\subset\Lambda(k_2,h)$. Then
\[
\theta_d(\uep;\ual,\ube)=\theta_h(\uep;\ual,\ube).
\]
\end{lemma}
\begin{proof}
Let $\rho\in(\uep)\subset\fS_h$,
and for simplicity, using the fact that $\supp(\rho)\leq d$, assume that $\rho$ is a permutation
that fixes pointwise the set $\set{d+1,\dots,h}\subset\set{1,\dots,h}$.
Then by the above combinatorial interpretation,
$\theta_h(\uep;\ual,\ube)$ is equal to the number of factorisations
$\rho=\sigma\cdot\tau$ with $\sigma\in(\ual)\subset\fS_h$ and $\tau\in(\ube)\subset\fS_h$.
In particular $N(\rho)=N(\sigma)+N(\tau)$, hence
by Lemma \ref{cor:geodesicpair} each cycle of $\sigma$ or of $\tau$ is contained in a cycle
of $\rho$. This imples that no cycle of $\sigma$ or $\tau$ of length $\ge2$ can contain some
element in $\set{d+1,\dots,h}$, as these elements are fixpoints of $\rho$, i.e. they belong to cycles
of length 1. In other words, both $\sigma$ and $\tau$ fix $\set{d+1,\dots,h}$ pointwise.
\end{proof}
\begin{nota}
\label{nota:productcoefficientbis}
For $\ual,\ube,\uep$ satisfying $N(\ual)+N(\ube)=N(\uep)$
we denote by $\theta(\uep;\ual,\ube)$ the common value of the expressions $\theta_d(\uep;\ual,\ube)$ for all
$d\ge\max(\supp(\ual),\supp(\ube),\supp(\uep))$.
\end{nota}

\begin{proposition}
\label{prop:phd}
The map $p^h_d$ is a map of $\Q$-algebras.
\end{proposition}
\begin{proof}
Let $\ual$ and $\ube$ satisfy $\supp(\ual),\supp(\ube)\le h$, and denote $k_1=N(\ual)$ and $k_2=N(\ube)$. We have
\[
\begin{split}
 p^h_d(\fg_{\ual}(h)\cdot\fg_{\ube}(h))
 &= p^h_d\pa{\sum_{\uep\in\Lambda(k_1+k_2,h)}\theta(\uep;\ual,\ube)\fg_{\uep}(h)}\\
 &= \sum_{\uep\in\Lambda(k_1+k_2,h)}\theta(\uep;\ual,\ube)p^h_d(\fg_{\uep}(h))\\
 &= \sum_{\uep\in\Lambda(k_1+k_2,d)}\theta(\uep;\ual,\ube)\fg_{\uep}(d)\\
 &=\fg_{\ual}(d)\cdot\fg_{\ube}(d)\\
 &=p^h_d(\fg_{\ual}(h))\cdot p^h_d(\fg_{\ube}(h)).
\end{split}
\]
Note in particular that if either inequality $\supp(\ual)>d$ or $\supp(\ube)>d$ holds, then $\fg_{\ual}(h)\cdot \fg_{\ube}(h)$ is a linear combination of generators $\fg_{\uep}(h)$ with $\supp(\uep)>d$, and we then have
$p^h_d(\fg_{\ual}(h)\cdot\fg_{\ube}(h))=0=p^h_d(\fg_{\ual}(h))\cdot p^h_d(\fg_{\ube}(h))$.
\end{proof}

\section{A smaller set of generators}
\label{sec:minimalgenerators}
Let again $m:=\floor{d/2}$. In this section we prove that the elements $\gamma_{2}(d),\dots,\gamma_{m+1}(d)$ suffice to generate $\A(d)$.
Set $h:=2d$ throughout the section and consider the map of $\Q$-algebras $p^h_d\colon \A(h)\to\A(d)$: it is surjective, hence any set of generators of $\A(h)$ maps along $p^h_d$ to a set of generators of $\A(d)$.

We can refine this observation as follows: since $\A(d)$ is concentrated
in norms $\leq d-1$ (indeed any permutation in $\fS_d$
has norm at most $d-1$), whenever $x_1,\dots,x_r\in\A(h)_{\le d-1}$ are elements that suffice to generate multiplicatively $\A(h)_{\le d-1}$, then $p^h_d(x_1),\dots,p^h_d(x_r)$ generate $\A(d)$ as a $\Q$-algebra.

Recall from Subsection \ref{subsec:notation} that for $1\leq i\leq d=h/2$ we have elements $\delta_{i}(h)\in\A(h)_i$; we will prove the following theorem.
\begin{proposition}
\label{prop:deltagenerators}
The elements $\delta_{1}(h),\delta_{2}(h),\dots,\delta_{d-1}(h)$ generate
$\A(h)_{\le d-1}$.
\end{proposition}
Suppose for a moment that Proposition \ref{prop:deltagenerators} holds. Then the images
\[
p^h_d(\delta_{1}(h)),p^h_d(\delta_{2}(h)),\dots,p^h_d(\delta_{d-1}(h))
\]
generate $\A(d)$ by the previous
remark. Note however that roughly half of these elements of $\A(d)$ are equal to zero:
more precisely we have $p^h_d(\delta_{i}(h))=\delta_{i}(d)$ for $1\leq i\leq m=\lfloor d/2\rfloor$,
and $p^h_d(\delta_{i}(h))=0$ for $i\geq m+1$.
So actually the elements $\delta_{1}(d),\delta_{2}(d),\dots,\delta_{m}(d)$
must suffice to generate $\A(d)$.

The elements $\delta_{1}(d),\delta_{2}(d),\dots,\delta_{m}(d)$
have norms $1,2,\dots,m$, and we know that we can generate them using
the set of generators $\gamma_{2}(d),\dots,\gamma_{d}(d)$ of $\A(d)$.
However we cannot use any generator $\gamma_{i}(d)$ of norm $i-1>m$ 
in order to generate the elements $\delta_{1}(d),\delta_{2}(d),\dots,\delta_{m}(d)$.
Hence the generators $\gamma_{2}(d),\dots,\gamma_{m+1}(d)\in\A(d)$ must suffice
to generate the elements $\delta_{1}(d),\delta_{2}(d),\dots,\delta_{m}(d)$,
which in turn suffice to generate the entire $\A(d)$.
This would conclude the claim at the beginning of the section. Using the discussion at the end of Subsection \ref{subsec:algebraicindependence}, we have completed the proof of Theorem \ref{thm:minimalgenerators}, assuming Proposition \ref{prop:deltagenerators}.

As a corollary of Proposition \ref{prop:deltagenerators}, we can also establish when the first minimal relations will occur.

\begin{corollary}\label{cor:firstminrel}
For $d=2m\geq 2$ there is a single smallest norm minimal relation in norm $m+1$. For $d=2m+1\geq 5$ there are two smallest norm minimal relations in norm $m+2$.
\end{corollary}
\begin{proof}
Assume first that $d=2m$.  We are interested in the kernel $I_d$ of the surjective homomorphism of graded $\Q$-algebras $\phi: \Q[X_1,\dots,X_m] \to \A(d)$ given by $X_i\mapsto \delta_i(d)$: a set of minimal relations for $\A(d)$ is given by a graded basis of $I_d/\Q[X_1,\dots,X_m]_+\cdot I_d$. From Proposition \ref{prop:algebraicindependence} we know that $(I_d)_{\leq m}=0$, implying that there are no minimal relations in norm $\le m$. In other words, the graded $\Q$-vector space $I_d$ is concentrated in norms $\ge m+1$; since $\Q[X_1,\dots,X_m]_+$ is concentrated in norms $\ge1$, we obtain that
the graded $\Q$-vector space $\Q[X_1,\dots,X_m]_+\cdot I_d$ is concentrated in norms $\ge m+2$, and this implies that the quotient map $I_d\to I_d/\Q[X_1,\dots,X_m]_+\cdot I_d $ restricts to an isomorphism in norm $m+1$:
\[
(I_d)_{m+1}\overset{\cong}{\to} \pa{I_d/\Q[X_1,\dots,X_m]_+\cdot I_d}_{m+1}.
\]

We are therefore led to compute $\dim_{\Q}((I_d)_{m+1})$. We know that $\phi$ is surjective in norm $m+1$, so \[\dim_{\Q}((I_d)_{m+1}) = \dim_{\Q}(\Q[X_1,\dots X_m]_{m+1}) - \dim_{\Q}(\A(d)_{m+1}),
\]
and we claim that this difference is equal to $1$ for $m\geq 1$.

Let $k$ denote the number of sequences $\ul=(\lambda_2,\lambda_3,\dots)$ satisfying $N(\ul)=m+1$.
The dimension $\dim_{\Q}(\Q[X_1,\dots X_m]_{m+1})$ is equal to the number of monomials in the variables $X_1,\dots, X_m$ of norm $m+1$: this is in turn equal to the number of sequences
$\ul=(\lambda_2,\lambda_3,\dots)$ satisfying $N(\ul)=m+1$ and $\lambda_i=0$ for all $i\ge m+1$.
The condition $N(\ul)=m+1$ immediately implies $\lambda_i=0$ for all $i\ge m+2$, and there is exactly one sequence $\ul$, namely $\one_{m+2}$, satisfying $N(\ul)=m+1$ but $\lambda_{m+1}\ge1$.
We conclude that
$k=1+\dim_{\Q}(\Q[X_1,\dots X_m]_{m+1})$.

The dimension $\dim_{\Q}(\A(d)_{m+1})$ is equal to the number of sequences $\ul = (\lambda_2,\lambda_3,\dots)$ satisfying  $N(\ul)=m+1$ and $\supp(\ul) \leq d=2m$.
We claim that there are only two sequences $\ul$ satisfying $N(\ul)=m+1$ and $\supp(\ul) \geq d+1=2m+1$; the claim would imply the equality $k=2+\dim_{\Q}(\A(d)_{m+1})$.
To prove the claim, note that if a sequence $\ul$ satisfies $N(\ul)=m+1$ and $\supp(\ul) \geq 2m+1$, then it must also satisfy 
\begin{align*}
     \supp(\ul)=\sum_{i=2}^{\infty}i\lambda_i \leq \sum_{i=2}^{\infty}(2i-2)\lambda_i = 2N(\ul)=2m+2.
\end{align*}
If the last inequality is in fact an equality, we must have $i\lambda_i=(2i-2)\lambda_i$ for all $i\ge2$, and we deduce the equality $\ul= (m+1,0,0,\dots)$. 

If instead $\supp(\lambda)=2m+1$, we have
\begin{align*}
 1+\sum_{i=2}^{\infty}i\lambda_i =1+\supp(\ul)= 2m+2 = 2N(\ul)= \sum_{i=2}^{\infty}(2i-2)\lambda_i,
\end{align*}
and hence $\sum_{i=2}^\infty (i-2)\lambda_i=1$. This implies $\lambda_3=1$ and $\lambda_i=0$ for all $i\ge4$, so that $\lambda = (m-1,1,0\dots)$.

Putting together the equalities 
$k=1+\dim_{\Q}(\Q[X_1,\dots X_m]_{m+1})$ and 
$k=2+\dim_{\Q}(\A(d)_{m+1})$, we obtain the desired equality
$\dim_{\Q}((I_d)_{m+1})=2-1=1$.

Assume now that $d=2m+1$ with $m\geq 2$; by Proposition \ref{prop:algebraicindependence} we obtain again that
$I_d$ is concentrated in norms $\ge m+1$. Let now $k$ denote the number
of sequences $\ul$ with $N(\ul)=m+1$.
The dimension $\dim_{\Q}(\Q[X_1,\dots,X_m]_{m+1})$ is equal to the number of sequences $\ul$ with $N(\ul)=m+1$ and $\lambda_i=0$ for all $i\ge m+1$; as before, we note that the condition $N(\ul)=m+1$ already implies $\lambda_i=0$ for all $i\ge m+2$, and there is exactly one sequence $\ul$ satisfying $N(\ul)=m+1$ and $\lambda_{m+2}\ge1$, namely $\one_{m+2}$.
We therefore have $k=1+\dim_{\Q}(\Q[X_1,\dots,X_m]_{m+1})$.

Similarly, the dimension $\dim_{\Q}(\A(d)_{m+1})$ is equal to the number of sequences $\ul$ with $N(\ul)=m+1$ and $\supp(\ul)\le d=2m+1$; as before, the inequality $\supp(\ul)\le 2 N(\ul)$ implies that any sequence $\ul$ with $N(\ul)=m+1$ also satisfies 
$\supp(\ul)\le 2m+2$, and the only sequence $\ul$ for which  $N(\ul)=m+1$ and $\supp(\ul)= 2m+2$ is $\ul=(m+1,0,0,\dots)$. We therefore have
$k=1+\dim_{\Q}(\Q[X_1,\dots,X_m]_{2(m+1)})$, and comparing with the other equality proved above we conclude that
$\dim_{\Q}(\Q[X_1,\dots,X_m]_{m+1})=
\dim_{\Q}(\A(d)_{m+1})$. The map $\phi$ is surjective in norm $m+1$, and the last equality of dimensions implies that $\phi$ is also injective in norm $m+1$, i.e. $(I_d)_{m+1}=0$.

We have thus shown that $I_d$ is concentrated in norms $\ge m+2$; as in the even case we conclude that the quotient map $I_d\to I_d/\Q[X_1,\dots,X_m]_+\cdot I_d$ is an isomorphism in norm $m+2$.

We are therefore led to compute $\dim_{\Q}((I_d)_{m+2})$, which by surjectivity of $\phi$ in norm $m+2$ is equal to the difference $\dim_{\Q}(\Q[X_1,\dots X_m]_{m+2}) - \dim_{\Q}(\A(d)_{m+2})$.
We claim that this difference is equal to $2$ for $m\geq 2$.

Let $k'$ denote the number of sequences $\ul$ satisfying $N(\ul)=m+2$.
The dimension $\dim_{\Q}(\Q[X_1,\dots X_m]_{m+2})$ is equal to the number of sequences $\ul$ satisfying $N(\ul)=m+2$ and $\lambda_i=0$ for all $i\ge m+1$. The condition $N(\ul)=m+2$ immediately implies $\lambda_i=0$ for all $i\ge m+3$: hence, for a sequence $\ul$ satisfying $N(\ul)=m+2$, we also have $\lambda_i=0$ for all $i\ge m+1$ unless one of the following occurs:
\begin{itemize}
    \item $\lambda_{m+2}\ge1$, in which case the equality $\sum_{i=2}^\infty(i-1)\lambda_i=m+2$ implies $\lambda_{m+2}=1$ and $\ul=\one_{m+2}$;
    \item $\lambda_{m+2}=0$ and $\lambda_{m+1}\ge1$, in which case the equality $\sum_{i=2}^\infty(i-1)\lambda_i=m+2$, together with the hypothesis $m\ge2$ (and hence $(m+2-1)\cdot 2> m+2$), implies $\lambda_{m+1}=1$ and $\ul=\one_2+\one_{m+1}$.
\end{itemize}
We thus obtain the equality $k'=2+\dim_{\Q}(\Q[X_1,\dots X_m]_{m+2})$.

The dimension $\dim_{\Q}(\A(d)_{m+2})$ is equal to the number of sequences $\ul$ satisfying $N(\ul)=m+2$ and $\supp(\ul)\le d=2m+1$. Let $\ul$ be a sequence with
 $N(\ul)=m+2$ but $\supp(\ul)\ge 2m+2$;
 as before we have $\supp(\ul)\le 2N(\ul)=2m+4$, and we can compute $0\le 2N(\ul)-\supp(\ul)=\sum_{i=2}^\infty(i-2)\lambda_i\le 2$; we have three cases:
 \begin{itemize}
     \item $\sum_{i=2}^\infty(i-2)\lambda_i=0$: in this case we have $\ul=(m+2,0,0,\dots)$;
     \item $\sum_{i=2}^\infty(i-2)\lambda_i=1$: in this case we have $\ul=(m,1,0,0,\dots)$;
     \item $\sum_{i=2}^\infty(i-2)\lambda_i=2$: in this case either $\lambda_4=0$ and $\ul=(m-2,2,0,0,\dots)$, or $\lambda_4\ge1$ and $\ul=(m-2,0,1,0,0,\dots)$; note that both cases are possible as we assume $m\ge2$.
 \end{itemize}
We thus obtain the equality $k'=4+\dim_{\Q}(\A(d)_{m+2})$, which together with the previous equality yields $\dim_{\Q}(\Q[X_1,\dots X_m]_{m+2}) - \dim_{\Q}(\A(d)_{m+2})=4-2=2$, as claimed.
\end{proof}

Using computer calculations to obtain presentations of $\A(d)$ for $d\leq 10$, we are led to the conjecture that there are no minimal relations in norm greater than $d$ (see Table \ref{tab:reltable}). The best upper bound we can prove is that there are no minimal relations in norm $\ge\lfloor\frac{3d}{2}\rfloor=d+m$. To see this, let $k\ge \lfloor\frac {3d}2\rfloor$; since $k\ge d$ we have $\A(d)_{k}=0$, hence $(I_d)_{k}=\Q[X_1,\dots,X_m]_{k}$, and we want to prove that we also have
\[
\pa{I_d\cdot\Q[X_1,\dots,X_m]_+}_{k}=\Q[X_1,\dots,X_m]_{k}.
\]
To see this, take a monomial $\prod_{i=1}^{m}X_i^{k_i}$ in norm $k$, with $k_i\geq 0$.

Choose the least $j$ such that $k_j\geq 1$, and factor the corresponding $X_j$ out of the product.
The norm of the monomial $\prod_{i=1}^{m}X_i^{k_i}/X_j$ is at least $k-m\ge d$, and hence the monomial $\prod_{i=1}^{m}X_i^{k_i}/X_j$ lies in $I_d$. We can thus exhibit our original monomial $\prod_{i=1}^{m}X_i^{k_i}$ as a product of an element $\prod_{i=1}^{m}X_i^{k_i}/X_j$ in $I_d$ and an element $X_j$ in $\Q[X_1,\dots,X_m]_+$.
\subsection{Strategy of proof of Proposition \ref{prop:deltagenerators}}
In the rest of the section we will only deal with the algebra $\A(h)$, so we will use again the short notation $\fg_{\ul}$ to denote $\fg_{\ul}(h)$.

We will prove by induction on $1\leq j\leq d-1$ the following statement: \emph{The elements
$\delta_{1},\delta_{2},\dots,\delta_{j}$ suffice to generate $\A(h)_{\le j}$.}

The case $j=1$ is easily solved: for $d\ge2$, $\A(h)_1$ has dimension 1 over $\Q$, with basis $\delta_{1}$.
The case $j=2$ can also be checked by hand: for $d\ge3$ we have $h=2d\ge6$; the equality $\delta_{1}^2=3\gamma_{3}+2\delta_{2}$, which is straightforward to check (see also Section \ref{sec:recursionformula}), ensures that $\delta_{1}$ and $\delta_{2}$ suffice
to generate multiplicatively the elements $1$, $\delta_{1}$, $\delta_{2}$ and $\gamma_{3}$, which form a basis of $\A(h)$ as a $\Q$-vector space in norms $\le 2$.

Suppose now by induction that
the elements $\delta_{1},\delta_{2},\dots,\delta_{j-1}$, having norms $\leq j-1$,
suffice to generate $\A(h)_{\le j-1}$, for some $j\ge2$.
It then suffices to prove that the element
$\gamma_{j+1}$ of norm $j$ can be generated using the element
$\delta_{j}$, which also has norm $j$, together with \emph{all} elements of norm $\leq j-1$ in $\A(h)$.

To make a more precise statement, for $1\leq s\leq j-1$ let $y_s$ be the product 
\[
y_s=\delta_{s}\cdot\gamma_{j+1-s}\in\A(h)_{j}.
\]
We also set $y_j=\delta_{j}$.
We will prove that $\gamma_{j+1}$ is a linear combination, over $\Q$, of the elements $y_1,\dots,y_j$.

We first express $y_s$ as a linear combination of some of the basis elements $\fg_{\ul}$.
For all $0\leq r\leq j-2$, let $\fg_{2^r,j+1-r}$ be the basis element of $\A(h)$
corresponding to the sequence
$\ul=(\l_2,\l_3,\dots)$ with $\l_2=r$, $\l_{j+1-r}=1$ and all other entries equal to 0. For $r=j-1$, we also set $\fg_{2^{j-1},2}:=j\delta_{j}\in\A(d)$.

For $0\le r\le j-2$, the element $\fg_{2^r,j+1-r}$ is the sum of all elements $[\sigma]\in\grad^N\Q[\fS_h]$
for $\sigma$ ranging among all permutations of $\fS_h$ with a cycle decomposition consisting
of $r$ transpositions, one additional $(j+1-r)$-cycle, and fixpoints. For $r=j-1$ we put an extra factor
$j$ to make the next formula look cleaner (and still be correct). Note that, for $r=0$, we have $\fg_{2^0,j+1-0}=\gamma_{j+1}$.

\begin{lemma}
\label{lem:ysexpansion}
For all $1\leq s\leq j-1$, we have the following formula for the expression
of $y_s$ as linear combination of the basis elements $\fg_{\ul}$:
\[
 y_s=
 \sum_{r=0}^{j-1}\frac{j+1-r}{j+1-s} \binom{j+1-s}{s-r}\fg_{2^r,j+1-r}.
\]
\end{lemma}
We will prove Lemma \ref{lem:ysexpansion} in Subsection \ref{subsec:ysexpansion}. For now, assume that it holds: then the $j$ elements
$(y_s)_{1\leq s\leq j}$ can be written as linear combinations of the $j$ elements
$(\fg_{2^r,j+1-r})_{0\leq r\leq j-1}$.

To prove the converse, i.e. that each element $\fg_{2^r,j+1-r}$ can be written as linear combination of the elements
$(y_s)_{1\leq s\leq j}$ (and in particular, for $r=0$, the element $\gamma_{j+1}$),
it suffices to check that the $(j\times j)$-matrix $A=(a_{r,s})_{0\leq r\leq j-1,1\le s\leq j}$
is invertible, where
\[
\begin{array}{rll}
a_{r,s}&=\frac{j+1-r}{j+1-s}\binom{j+1-s}{s-r} &\mbox{for } 0\leq r\leq j-1 \mbox{ and } 1\leq s\leq j-1,\\[5pt]
a_{r,j}&=0 &\mbox{for } 0\leq r\leq j-2,\\[5pt]
a_{j-1,j}&=\frac 1j. & 
\end{array}
\]
The values of $a_{r,j}$ for $0\leq r\leq j-1$ follow from the very definition of $y_j=\delta_{j}$
and $\fg_{2^{j-1},2}=j\delta_{j}$, hence $y_j=\frac 1j \fg_{2^{j-1},2}$.
Since $A$ has a column, the $j$\textsuperscript{th},  with only one entry different from $0$, we can check its invertibility
by passing to the corresponding minor of size $(j-1)\times(j-1)$, namely the matrix $B=(a_{r,s})_{0\leq r\leq j-2,1\leq s\leq j-1}$. We can now divide the $r$\textsuperscript{th} row of $B$ by $j+1-r$ and multiply the $s$\textsuperscript{th}
column of $B$ by $j+1-s$, to obtain the matrix $C=(c_{r,s})_{0\le r\le j-2,1\le s\le j-1}$, where
\[
    c_{r,s}=\binom{j+1-s}{s-r}.
\]
The fact that the matrix $C$ has a non-zero determinant will follow from applying Lemma~\ref{lem:pascalarray} to the mirrored matrix $\tilde C = (\tilde c_{k,l})_{1\leq k, l \leq j-1}$ where $\tilde c_{k,l} = c_{j-1-l,j-k} = \binom{k+1}{l+1-k}$; see Subsection \ref{subsec:pascalarray}.

\subsection{Proof of the Lemma \ref{lem:ysexpansion}}
\label{subsec:ysexpansion}
For all $1\leq s\leq j-1$ we have to prove the formula
\[
y_s=\delta_{s}\cdot\gamma_{j+1-s} = \sum_{r=0}^{j-1}\frac{j+1-r}{j+1-s} \binom{j+1-s}{s-r}\fg_{2^r,j+1-r}.
\]

First, we prove that $\delta_{s}\cdot\gamma_{j+1-s}$ is a linear combination of no other elements
$\fg_{\ul}$ than those of the form $\fg_{2^r,j+1-r}$. Let $\sigma,\tau\in\fS_h$ be permutations,
and assume that $\sigma$ has cycle decomposition with $s$ transpositions and $h-2s$ fixpoints, whereas
$\tau$ has a cycle decomposition with one $(j+1-s)$-cycle and $h-j-1+s$ fixpoints. Denote by
$\tr_1,\dots,\tr_s$
the transpositions occurring in $\sigma$, with $\tr_i=(n_in'_i)$ for distinct elements $n_1,n'_1,\dots,n_s,n'_s\in\set{1,\dots,h}$. Denote also by $c=(m_1\dots m_{j+1-s})$ the cycle of $\tau$.

Let $\rho=\sigma\tau$, and suppose that $N(\rho)=N(\sigma)+N(\tau)=j$. 
By Corollary \ref{cor:geodesicpair}, for all $1\leq i\leq s$, the transposition $\tr_i$ ``intersects''
the cycle $c$ in at most one element, i.e. at least one between $n_i$ and $n'_i$ is not contained
in $\set{m_1,\dots,m_{j+1-s}}$. We can assume without loss of generality that the transpositions
$\tr_1,\dots,\tr_t$ intersect $c$ in one element, and the transpositions $\tr_{t+1},\dots,\tr_s$
are disjoint from $c$, for some $0\leq t\leq s$.

We can multiply $\tau$ by the $s$ transpositions
constituting $\sigma$, one after the other, in order to obtain $\rho$ at the end; the order in which we multiply these transpositions is irrelevant,
as these transpositions commute with each other. The product $\tr_1\dots\tr_t\tau$ is a permutation
with one $(j+1-s+t)$-cycle and fixpoints. The cycle is supported on the set
$\set{m_1,\dots,m_{j+1-s}}\cup\set{n_1,n'_1,\dots,n_t,n'_t}$; in particular this cycle is disjoint
from all transpositions $\tr_{t+1},\dots,\tr_s$, and the conclusion is that $\rho$ is a permutation
with a cycle decomposition consisting of one $(j+1-s+t)$-cycle and $s-t$ transpositions, as wished: setting $r=s-t$, the element $[\rho]$ is one of the summands of $\fg_{2^r,j+1-r}$.

We notice two more phenomena. The first is that each transposition in the cycle decomposition
of $\rho$ is also a transposition $\tr_i$ in the cycle decomposition of $\sigma$, unless it coincides
with the cycle of $\tau$ (in which case we must have $s=j-1$, $t=0$ and thus $r=j-1$).

The second is that the elements $\set{m_1,\dots,m_{j+1-s}}\cup\set{n_1,n'_1,\dots,n_t,n'_t}$ composing
the \emph{long cycle} of $\rho$ occur in a controlled way in the cyclic order of this cycle: if we assume without
loss of generality that, for all $1\leq i\leq t$, the element $n_i$ belongs to $\set{m_1,\dots,m_{j+1-s}}$
and $n'_i$ does not, then the elements $n_i,n'_i$ occur consecutively in the long cycle of $\rho$.
Here by \emph{long cycle} we mean the cycle of $\rho$ containing the elements of the cycle of $\tau$. The conclusion
is that the elements $n_1,\dots,n_t$ are pairwise non-consecutive  in the cyclic order of the long cycle of $\rho$.

On the other hand, if $\rho$ is any permutation with a cycle decomposition consisting of 
$2^r$ transpositions and one $(j+1-r)$-cycle, with $(j+1-r)\ge3$, and if $n_1,\dots,n_t$ are non-consecutive elements
of this cycle, then we can define:
\begin{itemize}
    \item $\sigma$ to be the product of $r+t$ disjoint transpositions,
the first $r$ of which are the ones occurring in the cycle decomposition of $\rho$, the other $t$ of which have instead the form $(n_i,\rho^{-1}(n_i))$; 
 \item $\tau$ to be the permutation $\sigma\rho=\sigma^{-1}\rho$, which automatically contains only one non-trivial cycle, of length $(j+1-r-t)$.
\end{itemize}

We want now to solve the following quantitative problem. Let $1\leq s\leq j-1$, let $0\leq r\leq j-1$ and let
$\rho\in\fS_h$ be a permutation with a cycle decomposition consisting of
one $(j+1-r)$-cycle, $r$ transpositions and fixpoints; in the case $r=j-1$ we assume
that the cycle decomposition of $\rho$ has $j$ transpositions and fixpoints.
We want to compute the number $a_{r,s}$ of factorisations $\rho=\sigma\tau$ in $\fS_h$, with $\sigma$
consisting of $s$ transpositions and fixpoints, and $\tau$ consisting of a single $(j+1-s)$-cycle and fixpoints: the number 
$a_{r,s}$ is precisely the coefficient of $\fg_{2^r,j+1-r}$ in the formula for $\delta_{s}\cdot\gamma_{j+1-s}$,
as each such factorisation corresponds to a contribution $[\sigma]\cdot[\tau]=[\rho]$ in the computation
of $\fg_{2^s}\cdot\fg_{j+1-s}$.

By Corollary \ref{cor:geodesicpair}, for each such factorisation $\rho=\sigma\tau$,
the $(j+1-s)$-cycle of $\tau$ must be contained in a cycle of $\rho$: this
implies that $a_{r,s}=0$ whenever $r>s$, and this is compatible with the formula
$a_{r,s}=\frac{j+1-r}{j+1-s} \binom{j+1-s}{s-r}$, as the binomial coefficient vanishes whenever $r>s$.

From now on assume $0\le r\le s$.
We start with the case $r=s=j-1$, which is somehow exceptional. In this case $\rho$ consists
of $j$ transpositions, and we want to compute how many factorisations of the form $\rho=\sigma\tau$ exist,
with $\sigma$ consisting of $j-1$ transpositions and $\tau$ being a single transposition. The answer is $j$,
so that the coefficient of $\delta_{j}$ in $\delta_{s}\cdot\gamma_{j+1-s}$ is equal to $j$. If we plug this $j$
into the definition of $\fg_{2^{j-1},2}:=j\delta_j$, we obtain that the coefficient of
$\fg_{2^{j-1},2}$ in the product $\fg_{2^s}\cdot\fg_{j+1-s}$ is 1, and this is compatible
with the formula $a_{r,s}=\frac{j+1-r}{j+1-s} \binom{j+1-s}{s-r}$ also in this case.

From now on assume $0\le r\le j-2$. The permutation $\rho$ consists of
$r$ transpositions and one cycle of length $(j+1-r)\ge3$. If we write $\rho=\sigma\tau$ as above,
then the $r$ transpositions of $\rho$ should also be transpositions of $\sigma$, and the other
$s-r$ transpositions of $\sigma$ are uniquely determined by their intersections with the long
cycle of $\rho$. These intersections must be $s-r$ non-consecutive elements in a $(j+1-r)$-cycle,
so $a_{r,s}$ is equal to the number of ways to choos $s-r$ non-consecutive elements in a $(j+1-r)$-cycle.
For $r=s\le j-2$ we have $a_{r,s}=1$, and this is compatible with the formula
$a_{r,s}=\frac{j+1-r}{j+1-s} \binom{j+1-s}{s-r}$.

From now on we assume $0\le r<s$.
To compute $a_{r,s}$ we use a double counting argument: let $a'_{r,s}$ be the
number of ways to choose $s-r$ non-consecutive elements in a $(j+1-r)$-cycle, one of which is declared to
be the \emph{special element}.

On the one hand $a'_{r,s}=(s-r)a_{r,s}$, since we can first choose the $(s-r)$ non-consecutive elements, and then choose which of them is the special element.
On the other hand, we can first choose the special element (and we have $(j+1-r)$ ways to do it),
and then observe that any choice of the other $s-r-1$ elements determines a splitting of the number
$(j-r)-(s-r-1)=j-s+1$ as a sum of $s-r$ ordered, strictly positive integers. There are
$\binom{j-s}{s-r-1}$ ways to write $j-s+1$ as a sum of $s-r$ ordered, strictly positive integers,
and hence we get $a'_{r,s}=(j+1-r)\binom{j-s}{s-r-1}$.

Combining the two formulas we obtain also in the remaining cases
\[
a_{r,s}=\frac{j+1-r}{s-r}\binom{j-s}{s-r-1}=\frac{j+1-r}{j+1-s}\binom{j+1-s}{s-r}.
\]

\subsection{The lowest minimal relation in the even case}
Another consequence of Lemma~\ref{lem:ysexpansion} is the following mixed relation:
\begin{proposition}
    Let $m\ge1$; in $\A(2m)$ we have the following relation in norm $m+1$
    \begin{equation*}
        \sum_{s=1}^{m} (-1)^s s\cdot C_{m-s}\cdot \delta_s \cdot \gamma_{m+2-s} = 0.
    \end{equation*}
\end{proposition}
Note that the relation is nontrivial: by Proposition \ref{prop:sufficientgenerators} every $\delta_s$ can be rewritten in terms of the generators $\gamma_2,\ldots,\gamma_{m+1}$, and doing so yields a non-zero coefficient of $\gamma_2^{m+1}$, coming from the summand with $s=m$ (the other summands are multiples of a generator $\gamma_i$ with $i\ge3$). 
\begin{proof}
Let $j = m+1$ be fixed. From Lemma~\ref{lem:ysexpansion} we have
\begin{align*}
    \sum_{s=1}^{j-1} &(-1)^ss\cdot C_{j-1-s} \cdot\delta_s\cdot \gamma_{j+1-s}\\
    &= \sum_{s=1}^{j-1} (-1)^ss\cdot C_{j-1-s}
        \cdot\sum_{r=0}^{j-1}\frac{j+1-r}{j+1-s} \binom{j+1-s}{s-r}\cdot\fg_{2^r,j+1-r}\\
    &= \sum_{r=0}^{j-1} (j+1-r)\cdot \sum_{s=1}^{j-1} 
        \frac{(-1)^ss}{j+1-s} \binom{j+1-s}{s-r}C_{j-1-s}\cdot\fg_{2^r,j+1-r}.
\end{align*}
Since $\fg_{2^{j-2},3} = \fg_{2^{j-1},2} = 0$ in $\A(2j-2)$, it suffices to prove that
\begin{equation*}
    \mathcal{S}_r = \sum_{s=1}^{j-1} 
        \frac{(-1)^ss}{j+1-s} \binom{j+1-s}{s-r}C_{j-1-s}
        = 0, \qquad \mbox{for }0\leq r\leq j-3.
\end{equation*}
Starting from \cite[Eq. 16]{gauthier2011families} we have
\begin{equation*}
    0 = \sum_{k\geq 0} \binom{n-k}{k} C_{n-k-1} (-1)^k,  \qquad\mbox{for } n\geq 2.
\end{equation*}
On one hand, we can substitute $n=j+1-r$, $k=s-r$ and use the relation $C_{m+1} = \frac{2(2m+1)}{m+2}C_m$ to obtain
\begin{align}\label{eq:bernoulli1}
    0 &= \frac{1}{2}\sum_{s\geq r} \binom{j+1-s}{s-r} C_{j-s} (-1)^s \nonumber \\
    &= \sum_{s\geq r} \binom{j+1-s}{s-r} \frac{(2j-2s-1)}{j+1-s}C_{j-s-1} (-1)^s \nonumber \\
    &= (2j-1) A_r - 2 \mathcal{S}_r,
\end{align}
where
\begin{equation*}
    A_r = \sum_{s\geq r} \frac{(-1)^s}{j+1-s} \binom{j+1-s}{s-r} C_{j-s-1}.
\end{equation*}
On the other hand we can use the substitution $n=j-r$, $k=s-r$ to get
\begin{align} \label{eq:bernoulli2}
    0 &= \sum_{s\geq r} \binom{j-s}{s-r} C_{j-s-1} (-1)^s \nonumber \\
    &= \sum_{s\geq r} \frac{j+1-2s+r}{j+1-s} \binom{j+1-s}{s-r} C_{j-s-1} (-1)^s \nonumber \\
    &= (j+1+r)A_r -2\mathcal{S}_r.
\end{align}
Subtracting \eqref{eq:bernoulli2} from \eqref{eq:bernoulli1} yields $(j-r-2)A_r = 0$, which in return implies that  $A_r=\mathcal{S}_r=0$ for $0\leq r \leq j-3$.
\end{proof}

\subsection{Non-vanishing determinant of Pascal matrix}
\label{subsec:pascalarray}
We complete the proof of Proposition \ref{prop:deltagenerators} by showing that for $j\ge2$ the $(j-1)\times(j-1)$ \emph{Pascal matrix} $\tilde C = (\tilde c_{k,l})_{1\leq k, l \leq j-1}$ is invertible. Here is, as an example, the matrix $\tilde C$ for $j=7$:
\begin{align*}
    \begin{bmatrix}
    2 & 1  & 0  & 0 & 0 & 0 \\
    1 & 3  & 3  & 1 & 0 & 0 \\
    0 & 1  & 4  & 6 & 4 & 1 \\
    0 & 0  & 1  & 5 & 10& 10\\
    0 & 0  & 0  & 1 & 6 & 15\\
    0 & 0  & 0  & 0 & 1 & 7 \\
    \end{bmatrix}.
\end{align*}
Invertibility of $\tilde C$
is a consequence of the following lemma, which shows that the determinant of $\tilde C$ in fact the Catalan number $C_j=\frac{1}{j+1}\binom{2j}{j}$.

\begin{lemma}\label{lem:pascalarray}
Consider the array with the rows of Pascal's triangle given by
\begin{align*}
    X =
    \begin{bmatrix}
    2 & 1 & 0 & 0 & 0 & 0 & 0 & 0 & \cdots \\
    1 & 3 & 3 & 1 & 0 & 0 & 0 & 0 & \cdots \\
    0 & 1 & 4 & 6 & 4 & 1 & 0 & 0 & \cdots \\
    0 & 0 & 1 & 5 & 10 & 10 & 5 & 1 & \cdots \\
    \vdots & \vdots & \vdots & \vdots & \vdots
    & \vdots & \vdots & \vdots &\ddots
    \end{bmatrix},
\end{align*}
where $X_{ij} = \binom{i+1}{j-i + 1}$. Then the non-reduced echelon form of $X$ is the array
\begin{align*}
    Y=
    \begin{bmatrix}
    2 & 1 & 0 & 0 & 0 & 0 & 0 & 0 & \cdots \\
    0 & 5 & 6 & 2 & 0 & 0 & 0 & 0 & \cdots \\
    0 & 0 & 14 & 28 & 20 & 5 & 0 & 0 & \cdots \\
    0 & 0 & 0 & 42 & 120 & 135 & 70 & 14 & \cdots \\    \vdots & \vdots & \vdots & \vdots
    & \vdots & \vdots & \vdots & \vdots &\ddots
    \end{bmatrix},
\end{align*}
where $Y_{ij} = B_{i,j-i}$ are the numbers from Borel's triangle, given by
\begin{align*}
    B_{n,k} = \frac{1}{n+1} \binom{2n+2}{n-k} \binom{n+k}{n}.
\end{align*}
In particular the diagonal of $Y$ consists of the Catalan numbers, and the determinant of the top left $n\times n$ minor of $X$ is
\begin{align*}
    \det \left((X_{ij})_{1\leq i,j\leq n}\right) 
    = C_{n+1},
\end{align*}
where $C_n$ is the $n$\textsuperscript{th} Catalan number.
\end{lemma}
\begin{proof}
We prove by induction on $n\ge1$ that, after the $(n-1)$\textsuperscript{st} step of Gaussian elimination by rows on $X$, the $n$\textsuperscript{th} row is equal to the $n$\textsuperscript{th} row of $Y$. For $n=1$ this is clear since both arrays start with the row $(2,1,0,0,\ldots)$. Let now $n\ge1$ and assume that the first $(n-1)$ steps of Gaussian elimination transform the $n$\textsuperscript{th} row of $X$ into the $n$\textsuperscript{th} row of $Y$, and consider the $n$\textsuperscript{th} step of Gaussian elimination. By assumption the array $X$ has been transformed in such a way that the $n$\textsuperscript{th} and $(n+1)$\textsuperscript{st} rows look as follows
\begin{align*}
    \begin{array}{cccccccccccc}
        n\text{\textsuperscript{th} row: }  & 0 & \cdots & 0 & B_{n,0} & B_{n,1} & \ldots & B_{n,n} & 0 & 0 & 0 &  \cdots \\
        (n+1)\text{\textsuperscript{st} row: } & 0 & \ldots & 0 & \binom{n+2}{0} & \binom{n+2}{1} & \ldots & \binom{n+2}{n} & \binom{n+2}{n+1} & \binom{n+2}{n+2} & 0 & \cdots
    \end{array}
\end{align*}
where there are $n-1$ zeroes in the beginning of both rows. Hence Gaussian elimination will multiply the $(n+1)$\textsuperscript{st} row by $B_{n,0}$ and subtract the $n$\textsuperscript{th}
row from the $(n+1)$\textsuperscript{st}. This results in a row with entries
\[
X_{n+1,j}' =  B_{n,0} \binom{n+2}{j - n} - B_{n,j-n}.
\]
If $j-n> n$, we have the following equalities:
\begin{itemize}
    \item $B_{n,0} \binom{n+2}{j - n} - B_{n,j-n}=0=B_{n+1,j-n-1}$, if $j-n>n+2$;
    \item $B_{n,0} \binom{n+2}{j - n} - B_{n,j-n}=B_{n,0}=B_{n+1,n+1}=B_{n+1,j-n-1}$, if $j-n=n+2$;
        \item $B_{n,0} \binom{n+2}{j - n} - B_{n,j-n}=B_{n,0}=B_{n+1,n}=B_{n+1,j-n-1}$, if $j-n=n+1$.
\end{itemize}
If instead $j-n\le n$, using the explicit formula for the Borel numbers and setting $k = j-n$ we obtain
\begin{align*}
    B_{n,0} \binom{n+2}{k} - B_{n,k} 
    &= \frac{1}{n+1} \left[ \binom{n+2}{k}\binom{2n+2}{n} - \binom{2n+2}{n-k}\binom{n+k}{n}\right]\\
    &= \frac{(2n+2)!}{(n+1)!k!} \left[ \frac{1}{(n+2-k)!} -\frac{(n+k)!}{(n-k)!(n+2+k)!} \right]\\
    &= 
    \frac{(2n+2)!}{(n+1)!k!}\ 
    \frac{(n-k)!(n+2+k)! - (n+k)!(n+2-k)!}{(n+2-k)!(n-k)!(n+2+k)!} \\
    &= 
    \frac{(2n+2)!}{(n+1)!k!}\ 
    \frac{(4n+6)k(n+k)!}{(n+2-k)!(n+2+k)!}\\
    &= \frac{4n+6}{(2n+3)(2n+4)} \binom{2n+4}{n+2-k} \binom{n+k}{n+1} \\
    &= B_{n+1,k-1},
\end{align*}
where we have used the identity $(n-k)!(n+2+k)! - (n+k)!(n+2-k)! = (4n+6)k(n-k)!(n+k)!$, holding for $0\le k\le n$.

This shows that $Y$ is the echelon form of $X$, and moreover, it shows that the $(n-1)$\textsuperscript{st} step of Gaussian elimination multiplies the $n$\textsuperscript{th} row by $B_{n-1,0}=C_{n}$. Hence the determinant of any top left $n\times n$ minor of $X$ is given by
\begin{align*}
    \det \left((X_{ij})_{1\leq i,j\leq n}\right) 
    = \frac{\det(\left((Y_{ij})_{1\leq i,j\leq n}\right))}{C_1\cdots C_{n}} = \frac{C_1\cdots C_{n+1}}{C_1\cdots C_n} = C_{n+1}.
\end{align*}
\end{proof}

\section{A recursion formula for multiplication}
\label{sec:recursionformula}
In this subsection we provide a recursive formula for computing the product $\fg_{\ual} \fg_{\ube}$ in $\A(d)$, using only the sequences $\ual = (\alpha_2,\alpha_3,\dots)$ and $\ube=(\beta_2,\beta_3,\dots)$ rather than the permutations they give rise to. This recursive formula is used to compute the multiplication tables for $\A(d)$ appearing in Section \ref{sec:smallcases}. We remark that a similar formula, in the case in which either $\ual$ or $\ube$ is of the form $\gamma_i$ for some $i\ge2$, was given in \cite[Lemma 2.2]{Hikita}.

For $\ual$, $\ube$ and $\uep$
as in Notation \ref{nota:productcoefficientbis}
we want to compute the product $\fg_{\ual}\fg_{\ube}$, and in particular determine the coefficient
$\theta(\uep;\ual,\ube)$
of the basis vector $\fg_{\uep}$ in the standard expression of this product.

 In the following we give recursive formulas to compute all coefficients $\theta(\uep;\ual,\ube)$. The basis of the recursion is $\theta(\uep;\uep,\uz)=\theta(\uep;\uz,\uep)=1$, where $\uz$ denotes the zero sequence: notice indeed that $\fg_{\uz}$ is the neutral element of $\A(d)$.

\subsection{The case of a single target cycle}
We first focus on the case in which $\uep=\one_\ell$ corresponds to permutations with a single cycle of length $2\le\ell\le d$; in this subsection we denote by $\ual$ and $\ube$ the two sequences corresponding to the two factors.
\begin{lemma}
\label{lem:nonvanish}
Let $\ual$ and $\ube$ be two sequences, and assume that $\theta(\one_\ell;\ual,\ube)$ does not vanish;
then both of the following are satisfied:
\begin{itemize}
 \item $N(\ual)+N(\ube)=N(\one_\ell)=\ell-1$;
 \item $\supp(\ual)\le \ell$ and $\supp(\ube)\le \ell$.
\end{itemize}
\end{lemma}
\begin{proof}
By definition of the product in the algebra $\A(d)$,
the coefficient $\theta(\one_\ell;\ual,\ube)$ does not vanish if and only if there exist permutations $\sigma$ of type $\ual$ and $\tau$ of type $\ube$ such that $N(\sigma\tau)=N(\sigma)+N(\tau)$ and such that $\sigma\tau$ is a permutation of type $\one_\ell$, i.e. it consists a single $\ell$-cycle.
 
 Assuming the existence of $\sigma$ and $\tau$ with the mentioned properties, we immediately obtain the equality $N(\ual)+N(\ube)=N(\one_\ell)=\ell-1$. 
 
 We then observe that, by Corollary \ref{cor:geodesicpair}, each non-trivial cycle of $\sigma$ and of $\tau$ is contained in the unique cycle $c$ of $\sigma\tau$, and thus $\supp(\sigma\tau)=\ell$ is an upper bound for both $\supp(\sigma)$ and $\supp(\tau)$.
 \end{proof}

\begin{nota}
For $\ell\ge2$ and any $\ual$, $\ube$ we denote
\[
\Theta(\one_\ell;\ual,\ube):=(\ell-1)!\ \theta(\one_\ell;\ual,\ube).
\]
\end{nota}
In the following we give a recursion for the numbers $\Theta(\one_\ell;\ual,\ube)$; the numbers of interest
$\theta(\one_\ell;\ual,\ube)$ can then be computed by dividing by $(\ell-1)!$.

Fix $\ell\ge2$ and $\ual$ and $\ube$ satisfying the hypotheses of Lemma \ref{lem:nonvanish}. 
Note that at least one of the inequalities
$\supp(\ual)\le \ell$ and $\supp(\ube)\le \ell$ must be strict: indeed if $\supp(\ual)= \ell$ and $\supp(\ube)= \ell$,
using that $N(\ual)\ge \frac 12\supp(\ual)$ and $N(\ube)\ge \frac 12\supp(\ube)$, we would not have $N(\ual)+N(\ube)=\ell-1$.

In the following assume without loss of generality that $\supp(\ual)<\ell$. For $i\ge2$ denote by $\del_i\ube$
the sequence
\[
 \del_i\ube=(\beta_2,\beta_3,\dots,\beta_{i-1}+1,\beta_i-1,\beta_{i+1},\dots).
\]
Here we use the following conventions:
\begin{itemize}
 \item if $i=2$, then we just take $(\beta_2-1,\beta_3,\beta_4,\dots)$;
 \item if the obtained sequence contains some negative number, then later the corresponding factor/summand will be automatically declared to be 0.
\end{itemize}
\begin{proposition}\label{prop:cycle_recursion}
Let $2\le\ell\le d$, and let $\ual$, $\ube$ be such that $N(\ual)+N(\ube)=\ell-1$, $\supp(\ual)<\ell$ and $\supp(\ube)\le \ell$; then we have
\[
 (\ell-\supp(\ual))\ \Theta(\one_\ell;\ual,\ube)=\ell\sum_{i=2}^\infty(i-1)\ (\del_i\ube)_{i-1}\ \Theta(\one_{\ell-1};\ual,\del_i\ube).
\]
Here $(\del_i\ube)_{i-1}$ denotes the $(i-1)$\textsuperscript{st} component of the sequence $\del_i\ube$, i.e. the one that increases by 1 when passing from $\ube$ to $\del_i\ube$. In the case $i=2$ we declare $(\del_i\ube)_{i-1}:=\ell-\supp(\ube)+1$.
\end{proposition}
\begin{proof}
We use a double counting argument. We count the number of quadruples $(\sigma,\tau,\rho,j)$ such that all of the following hold:
\begin{enumerate}
    \item $\sigma\in\fS_d$ is of type $\ual$, $\tau\in\fS_d$ is of type $\ube$;
    \item $\rho\in\fS_d$ is of type $\one_\ell$, and the support of $\rho$, considered as a subset of $\set{1,\dots,d}$, is precisely the set $\set{1,\dots,\ell}$;
    \item $\sigma\tau=\rho$;
    \item $j$ is one of the fixpoints of $\sigma$ contained in $\set{1,\dots,\ell}$.
\end{enumerate}
In the first count, we first choose $\rho$: there are $(\ell-1)!$ ways in order to ensure (2). We then choose $\sigma$ and $\tau$: there are $\theta(\one_\ell;\ual,\ube)$ ways in order to ensure (1) and (3). Finally, we choose $j$: there are $(\ell-\supp(\ual))$ ways, as we have to avoid the elements of $\set{1,\dots,\ell}$ belonging to the support of $\sigma$. Thus the first count gives the left hand side.

In the second count, we first choose $j$ in $\set{1,\dots,\ell}$: there are $\ell$ ways.

We then note the following: for any triple $(\sigma,\tau,\rho)$ forming an allowed quadruple together with $j$,
the element $j$ cannot be a fixpoint of $\tau$: indeed by (4) $j$ is a fixpoint of $\sigma$, but by (2) $j$ is not a fixpoint of $\rho$, and by (3) we must have $\rho=\sigma\tau$. For a triple $(\sigma,\tau,\rho)$ as above we denote by $\tau^{(j)}$ the permutation $\tau (j,\tau^{-1}(j))$, obtained by 
composing $\tau$ and the transposition
swapping $j$ and $\tau^{-1}(j)$. Note also that (3) and (4) also imply that $\tau^{-1}(j)=\rho^{-1}(j)$. We also denote $\rho^{(j)}=\rho(j,\tau^{-1}(j))$; note that (3) is equivalent to the equality $\rho^{(j)}=\sigma\tau^{(j)}$, and that both $\rho^{(j)}$ and $\tau^{(j)}$ have $j$ as fixpoint. Moreover (2) implies that $\rho^{(j)}$ is of type $\one_{\ell-1}$, whereas (1) implies that $\tau^{(j)}$ is of type $\del_i\ube$ for some $i\ge2$ (this $i$ is the length of the cycle of $\tau$ containing $j$).

We can now continue the second count as follows. After $j$ is chosen, we choose
$i\ge2$, and attempt to count triples $(\sigma,\tau,\rho)$ forming an allowed quadruple together with $j$, and such that $j$ belongs to a cycle of $\tau$ of length $i$; at the end we will sum over $i\ge2$. Once also $i$ has been fixed, we first choose $\rho^{(j)}$: there are $(j-2)!$ ways, since it has to be a permutation of type $\one_{\ell-1}$ with support the set $\set{1,\dots,\ell}\setminus\set{j}$. Next, we choose permutations $\sigma$ of type $\ual$ and $\tau^{(j)}$ of type $\del_i\ube$ such that $\sigma\tau^{(j)}=\rho^{(j)}$: there are $\theta(\one_{\ell-1};\ual,\del_i\ube)$ possibilities, and an application of Corollary \ref{cor:geodesicpair} ensures that $j$ is a fixpoint for both $\sigma$ and $\tau^{(j)}$. Next, we choose the element $\tau^{-1}(j)$: it must belong to
the set $\set{1,\dots,\ell}$, and it must belong to a cycle of $\tau^{(j)}$ of length $i-1$, so there are $(i-1)(\del_i\ube)_{i-1}$ possibilities. Thus the second count gives the right hand side.
\end{proof}
We observe that Proposition \ref{prop:cycle_recursion} implies that the converse of Lemma \ref{lem:nonvanish} holds: if $\ual$ and $\ube$ are two sequences satisfying both conditions in Lemma \ref{lem:nonvanish}, then $\theta(\one_\ell;\ual,\ube)>0$. We will however not use this remark.

More importantly, the hypotheses of Proposition \ref{prop:cycle_recursion} imply that the factor $(\ell-\supp(\ual))$ in the left hand side is non-zero, so we are allowed to divide by it and thus obtain a formula expressing $\Theta(\one_\ell;\ual,\ube)$ in terms of other coefficients $\Theta(\one_{\ell-1};\ual,\del_i\ube)$, related to sequences of strictly smaller support: this allows to compute all coefficients $\Theta(\one_\ell;\ual,\ube)$ recursively.

\subsection{Reduction to a single target cycle}
Consider now the case in which $\uep$ is a not of the form $\one_\ell$, and fix an index $\nu\ge2$ for which $\epsilon_\nu\ge1$.
\begin{defn}
A $\uep$-decomposition of $\ual$ is a couple of sequences $(\uA,\uA')$ satisfying the following properties:
\begin{itemize}
 \item $\uA+\uA'=\ual$ (here the sum of sequences is meant componentwise);
 \item $\supp(\uA)\le \nu$ (this implies in particular that $A_i=0$ for $i>\nu$).
\end{itemize}
If $(\uA,\uA')$ and $(\uB,\uB')$ are $\uep$-decompositions of $\ual$ and $\ube$ respectively, we say that they are \emph{$\uep$-compatible} if $N(\uA)+N(\uB)=\nu-1$.
\end{defn}
We observe that if $(\uA,\uA')$ and $(\uB,\uB')$ are $\uep$-compatible $\uep$-decompositions of $\ual$ and $\ube$,
then by the converse of Lemma \ref{lem:nonvanish} we have $\theta(\one_\nu;\uA,\uB)>0$. We will however not use this remark.

\begin{proposition}\label{prop:general_recursion}
 Let $\ual,\ube,\uep$ be sequences with $N(\ual)+N(\ube)=N(\uep)$ and such that $\max(\supp(\ual),\supp(\ube),\supp(\uep))\le d$. Let $\nu\ge2$ be an index such that $\epsilon_\nu\ge1$. Then we have
 \[
  \theta(\uep;\ual,\ube)=\sum_{(\uA,\uA'),(\uB,\uB')} \theta(\one_\nu;\uA,\uB)\,\theta(\uep-\one_\nu;\uA',\uB'),
 \]
where the sum is extended over all couples of $\uep$-compatible $\uep$-decompositions of $\ual$ and $\ube$, and where the difference of sequences $\uep-\one_\nu$ is computed coordinatewise.
\end{proposition}
\begin{proof}
Fix a permutation $\rho\in\fS_d$ of type $\uep$, such that the set $\set{1,\dots,\nu}$ is the support of a cycle $\mathbf{c}$ of $\rho$ of length $\nu$. Denote $\hat\rho=\mathbf{c}\in\fS_d$ and $\hat\rho'=\rho\mathbf{c}^{-1}\in\fS_d$.

If $\sigma$ and $\tau$ are permutations of types $\ual$ and $\ube$ respectively, satisfying $\rho=\sigma\tau$, then by Corollary \ref{cor:geodesicpair} each cycle of $\sigma$ and each cycle of $\tau$ is contained in a cycle of $\rho$.
We can then factor $\sigma=\hat\sigma\hat\sigma'$, where $\hat\sigma$ is the product of all cycles of $\sigma$ with support contained in $\set{1,\dots,\nu}$, and $\hat\sigma'$ is the product of all cycles of $\sigma$ with support contained in $\set{\nu+1,\dots,d}$. Similarly, we can factor $\tau=\hat\tau\hat\tau'$. 
We then have $\hat\sigma\hat\tau=\hat\rho$ and $\hat\sigma'\hat\tau'=\hat\rho'$. If we denote by $\uA,\uA',\uB,\uB'$ the cycle types of $\hat\sigma,\hat\sigma',\hat\tau,\hat\tau'$ respectively, then we have that $(\uA,\uA')$ and $(\uB,\uB')$ are $\uep$-compatible $\uep$-decompositions of $\ual$ and $\ube$ respectively.

On the other hand, we can make the following count. We first fix 
$\uep$-compatible $\uep$-de\-com\-po\-si\-tions $(\uA,\uA')$ and $(\uB,\uB')$
of $\ual$ and $\ube$ respectively, and at the end we will sum over all choices of $(\uA,\uA')$ and $(\uB,\uB')$.
We then choose permutations $\hat\sigma$ of type $\uA$ and $\hat\tau$ of type $\uB$ with $\hat\rho=\hat\sigma\hat\tau$: we have $\theta(\one_\nu;\uA,\uB)$ possibilities, since $\hat\rho$ is a permutation of type $\one_\nu$. Finally, we choose permutations $\hat\sigma'$ of type $\uA'$ and $\hat\tau'$ of type $\uB'$ with $\hat\rho'=\hat\sigma'\hat\tau'$: we have $\theta(\uep-\one_\nu;\uA',\uB')$ possibilities, since $\hat\rho'$ is a permutation of type $\uep-\one_\nu$. This count gives the right hand side in the formula.
\end{proof}
Whenever $\uep$ is not of the form $\one_\ell$,
Proposition \ref{prop:general_recursion} gives a formula to express $\theta(\uep;\ual,\ube)$ in terms of other coefficients $\theta(\gamma_\nu;\uA,\uB)\,\theta(\uc-\gamma_\nu;\uA',\uB')$ related to sequences of strictly smaller support: this allows to compute all coefficients $\theta(\uep;\ual,\ube)$ recursively.

Our computer-aided computations in Section \ref{sec:smallcases} use the formula from Proposition \ref{prop:general_recursion} with the \emph{smallest} index $\nu\ge2$ such that $\epsilon_{\nu}\ge1$.

\section{Small case minimal presentations}
\label{sec:smallcases}
For $1\leq d\leq 10$ we are able to produce minimal presentations of $\A(d)$.
In Appendix~\ref{app:algo} we describe the main algorithm used to compute the minimal presentations. The algorithm is implemented with symbolic computation in Python using the SymPy package \cite{SymPy}, and the code, including the datasets generated during and analysed during the current study are available in the \href{https://github.com/AlexanderChristgau/generators-and-relations}{GitHub} repository, is publicly available\footnote{\url{https://github.com/AlexanderChristgau/generators-and-relations}}. As a result we have the following proposition on minimal presentations based on the generators $(x,y,z,w) =(\gamma_2,\gamma_3,\gamma_4,\gamma_5)$.
\begin{proposition}
For $1\leq d \leq 8$, we have the following presentations of $A(d)$:
\begin{align*}
    \A(1) \cong & \Q\\
    \A(2) \cong & \Q[x]/(x^2) \\
    \A(3) \cong & \mathbb{Q}[x]/(x^3) \\ 
    \A(4) \cong & \mathbb{Q}[x,y]/(x^3-4xy,x^4,y^2) \\
    \A(5) \cong & \mathbb{Q}[x,y]/(x^4-5x^2y,x^4-25y^2, x^5) \\
    \A(6) \cong & \mathbb{Q}[x,y,z]/(x^4 - 11x^2y +24 xz +6y^2,\\
        & \phantom{\mathbb{Q}[x,y,z]/(} x^5-6x^3y,x^5-36x^2z, x^5-216yz,x^6,z^2) \\
    \A(7) \cong & \mathbb{Q}[x,y,z]/(x^5 - 13x^3y + 28 x^2z +14 xy^2, \\
        & \phantom{\mathbb{Q}[x,y,z]/(} 
        11x^5 -129x^3y +280x^2z +588yz, \: x^6 - 7x^4y, \\ 
        & \phantom{\mathbb{Q}[x,y,z]/(} x^6-49x^3z, \: x^6-343y^3, \: x^6-2401z^2,\: x^7) \\
    \A(8) \cong & \mathbb{Q}[x,y,z,w]/
        (x^5 - 21x^3y + 92x^2z +54 xy^2-240xw-96yz, \\
        & \phantom{\mathbb{Q}[x,y,z,w]/(} 
        x^6  - 15x^4y + 32x^3z  24 x^2y^2, \\
        & \phantom{\mathbb{Q}[x,y,z,w]/(}
        x^6 + 9x^4y   -304x^3z  +1440x^2w  -96 y^3, \\
        & \phantom{\mathbb{Q}[x,y,z,w]/(}
        x^6 + 87x^4y  -1232x^3z  +5472x^2w + 5376 yw, \\
        & \phantom{\mathbb{Q}[x,y,z,w]/(}
        17x^6  -135x^4y + 784x^3z +5760x^2w + 3584 z^3, \\
        & \phantom{\mathbb{Q}[x,y,z,w]/(}
        x^7-8x^5y, x^7-64x^4z, x^7-512x^3w, x^7-32768zw, \\ 
        & \phantom{\mathbb{Q}[x,y,z,w]/(} x^8,w^2)
\end{align*}
Furthermore, minimal presentations for $\A(9)$ and $\A(10)$ are given in Appendix~\ref{app:A9A10}.
\end{proposition}
For $d\ge1$ and $n\ge1$,
let $\fr_{d,n}$ denote the number of minimal relations of $\A(d)$ in norm $n$.
These presentations lead us to construct the following table with the numbers $\fr_{d,n}$,
along with the total number of necessary relations (the entries under $\Sigma$) for a given $A(d)$.

\begin{align}\label{tab:reltable}
\begin{array}{c| c| c | c | c | c | c | c | c | c | c | c || c}
d \, \backslash \, n & 1 & 2 & 3 & 4 & 5 & 6 & 7 & 8 & 9 & 10 & 11 & \Sigma \\ \hline
1 & 0 & 0 & 0 & 0 & 0 & 0 & 0 & 0 & 0 & 0 & 0 & 0 \\ \hline
2 & 0 & 1 & 0 & 0 & 0 & 0 & 0 & 0 & 0 & 0 & 0 & 1 \\ \hline
3 & 0 & 0 & 1 & 0 & 0 & 0 & 0 & 0 & 0 & 0 & 0 & 1 \\ \hline 
4 & 0 & 0 & 1 & 2 & 0 & 0 & 0 & 0 & 0 & 0 & 0 & 3 \\ \hline 
5 & 0 & 0 & 0 & 2 & 1 & 0 & 0 & 0 & 0 & 0 & 0 & 3 \\ \hline
6 & 0 & 0 & 0 & 1 & 3 & 2 & 0 & 0 & 0 & 0 & 0 & 6\\ \hline 
7 & 0 & 0 & 0 & 0 & 2 & 4 & 1 & 0 & 0 & 0 & 0 & 7\\ \hline
8 & 0 & 0 & 0 & 0 & 1 & 4 & 4 & 2 & 0 & 0 & 0 & 11\\ \hline
9 & 0 & 0 & 0 & 0 & 0 & 2 & 5 & 5 & 1 & 0 & 0 & 13\\ \hline
10 & 0 & 0 & 0 & 0 & 0 & 1 & 4 & 7 & 4 & 2 & 0 & 18\\ \hline
11^* & 0 & 0 & 0 & 0 & 0 & 0 & 2 & 6 & 8 & 4 & 1 & 21\\ \hline
\end{array}
\end{align}
$^*$Based on numerical approximation when computing the nullspace $\operatorname{Null}(A)$ described in Appendix~\ref{app:algo}.

Note that the entries $\fr_{2k,k+1}=1$ for $k\ge1$ and $\fr_{2k+1,k+2}=2$ for $k\ge2$ follow from Corollary \ref{cor:firstminrel}.

\subsection{Conjectures}
Based on the evidence from the table \ref{tab:reltable}, we make the following conjecture.

\begin{conjecture} The following statements hold:
\begin{itemize}
    \item For all $d\ge 2$ the algebra $\A(d)$ can be described in terms of relations of norm at most $d$.\vspace{3pt}
\item For every $d_0, n_0\ge0$ the numbers $\fr_{d_0+2k,n_0+k}$ stabilise for $k\to\infty$.\vspace{3pt}
\item For every $d_0, n_0\ge0$ the numbers $\fr_{d_0+2k,n_0+2k}$ stabilise for $k\to\infty$.\vspace{3pt}
\item For $d\ge3$ odd, the total number of minimal relations in $\A(d)$ is
\[\frac{(d-1)(d-3)}{4}+1.\]
\item For $d\ge4$ even, the total number of minimal relations in $\A(d)$ is
\[\frac{d(d-2)}{4}-\frac{d-6}{2}.\]
\end{itemize}
\end{conjecture}

\printbibliography
\appendix

\section{Minimal presentation algorithm}\label{app:algo}
The formula in Proposition \ref{prop:cycle_recursion} allows us to recursively compute any coefficient of the form $\theta(\one_\ell;\ual,\ube)$ efficiently. In order to extend the computation via the formula in Proposition \ref{prop:general_recursion}, we need to generate all $\uep$-decompositions of $\ual$. This can be done by iterating over the Cartesian product $\prod_{i=2}^\nu \N_{\leq \min(\alpha_i,\nu)}$ and checking the condition $|\uA|\le \nu$. 

Once the above procedure is implemented, a general procedure for computing basis expansions of products can be implemented by finding the coefficient for each basis element. The basis elements are in correspondence with partitions and thus we can borrow a preexisting fast algorithm for generation of partitions, e.g. \texttt{accel\_asc} from \cite{kelleher2009generating}. Let $m=\floor{d/2}$; since $\gamma_2,\ldots,\gamma_{m+1}$ are minimal generators for $\A(d)$, one can compute basis representations of all monomials to find a minimal set of relations:

\vspace{7pt}
{ \small 
\begin{mdframed}[roundcorner=10pt,frametitle={Computation of minimal presentation of $\A(d)$
}]
\begin{enumerate}
    \item For each norm $2\leq n\leq \floor{3d/2}=d+m$, find a sufficient set of relations in norm $n$ with the following procedure:
    \begin{itemize}
        \item Let $B_n$ be the set of all basis elements with norm $n$.
        
        \item Let $M_n$ be the set of monomials of $\gamma_2,\ldots,\gamma_{m+1}$ with norm $n$.
        
        \item Compute the $B_n\times M_n$ matrix $A^n$ defined by setting $A_{bm}^n$ to be the coefficient of $b\in B_n$ in the basis expansion of $m\in M_n$. 
        
        \item A sufficient set of relations in norm $n$ can be found by computing a basis of $\operatorname{Null}(A_n)$. 
        Compute a set $R_n$ consisting of such basis vectors.
    \end{itemize}
    Combined together, $R = R_1 \cup \cdots \cup R_{d+m}$ correspond to a sufficient set of relations in $\A(d)$.
    
    \item Generate every possible \textit{lifted} relation obtained by multiplying any relation in $R$ by any generator:
    \begin{itemize}
        \item For each $1\leq i \leq m$ and each $1\leq k \leq d+m-i$ construct an $M_{k+i}\times M_k$ conversion matrix $C^{i,k}$ with entries 
        \[
            C_{m,m'}^{i,k} = \one(m=\gamma_i m'), \qquad m \in M_{k+i}, m' \in M_k.
        \]
        \item Construct the set of lifted relations in norm $n$ by 
        $$L_n = \cup_{i+k=n} \{C^{i,k}r \colon r \in R_k\}.$$ 
    \end{itemize}
    
    \item Thin out the relations by first concatenating the relations into a matrix $W_n = \begin{pmatrix} L_n & R_n \end{pmatrix}$ for each norm $n$. Then the columns of $W_n$ containing the pivots after Gaussian elimination will constitute a minimally sufficient set of relations in norm $n$. Collectively these relations correspond to a minimal representation of $\A(d)$ with generators $\gamma_2,\ldots,\gamma_{m+1}$.
\end{enumerate}
\end{mdframed}
}
\vspace{7pt}
\textbf{Remark.} 
A first approach is to store the basis elements as lists of permutations with e.g. the SymPy \texttt{permutation} class. This allows for a more direct implementation of multiplication based on the presentation $\A(d) \cong (\grad^N\Q[\fS_d])^{\fS_d}$. However, this approach becomes computationally infeasible at around $d\approx 8$. The power of the recursion formula is that it makes it possible to store the basis element $\fg_{\ul}$ as a simple tuple $\ul$. With this reduction, there are no problems with memory usage and the computational bottleneck turns out to be symbolically computing $\operatorname{Null}(A_n)$ rather than $A_n$. The current computations were carried out on an M1 processor with approximate runtimes of $0.1, 0.3$ and $1189$ seconds for $d=8,9,10$, respectively. Thus it seems unclear if $d>10$ can be computed with high performance computing using the current implementation.

\section{Minimal presentations for $\A(9)$ and $\A(10)$}\label{app:A9A10}
\setcounter{MaxMatrixCols}{40}
We present the relations in each norm as arrays. For $\A(9)$ the top row denotes the monomials in the given norm and each subsequent row denotes the coefficients of a relation. For $\A(10)$ the arrays are transposed due to their large width. The generators are denoted by $(a,b,c,d,e) =(\gamma_2,\gamma_3,\gamma_4,\gamma_5, \gamma_6)$. 

\subsection{A minimal set of relations in $\A(9)$} 
In norm $6$ and $7$ we have
\[ \begin{matrix}
a^{6} & a^{4}b & a^{3}c & a^{2}b^{2} & a^{2}d & abc & b^{3} & bd & c^{2} \\
-1 & 24 & -104 & -87 & 270 & 216 & 18 & 0 & 0 \\
-11 & 258 & -1120 & -837 & 2910 & 1488 & 0 & 2160 & 864 \\
\end{matrix} \]

\[ \begin{matrix}
a^{7} & a^{5}b & a^{4}c & a^{3}b^{2} & a^{3}d & a^{2}bc & ab^{3} & abd & ac^{2} & b^{2}c & cd \\
1 & -17 & 36 & 36 & 0 & 0 & 0 & 0 & 0 & 0 & 0 \\
1 & -11 & -30 & 0 & 270 & 162 & 0 & 0 & 0 & 0 & 0 \\
251 & -3895 & 16284 & 0 & -50328 & 0 & 0 & 326592 & 0 & 0 & 0 \\
37 & -1001 & 9636 & 0 & -37800 & 0 & 0 & 0 & 0 & 46656 & 0 \\
85 & -857 & -2652 & 0 & 26712 & 0 & 0 & 0 & 0 & 0 & 373248 \\
\end{matrix} \] \phantom{X} \\
The relations in norm 8 and 9 may be summarized by
\begin{align*}
    a^8 = 9a^6b = 81 a^5c = 729 a^4d = 59049 b^2d = 531441 d^2, \\
    a^9 = 0.
\end{align*}

\subsection{A minimal set of relations in $\A(10)$}
In norm $6,7$ and $8$ we have
\[ \begin{matrix}
a^{6} &1\\
a^{4}b&-34\\
a^{3}c&224 \\
a^{2}b^{2} & 207 \\
a^{2}d&-1170\\
abc& -1136\\
ae&3360\\
b^{3}& -108\\
bd&1200\\
c^{2}& 480
\end{matrix} \]

\[ \begin{matrix}
a^{7} & -1 &-3   &  -7   &  -59 \\
a^{5}b &27 &86   &   229  &  1283\\
a^{4}c&-116 &-468  &  -1652   &  -2404 \\
a^{3}b^{2}&-126 &-393  &  -1152   &  -3654\\
a^{3}d &300&1950   &  9450   &  -18450 \\
a^{2}bc &360&1420   &   6080   &  -10040\\
a^{2}e&60&-4200  &  -29460   &  140700 \\
ab^{3}&0&0   &  0   &  0\\
abd& 0& 0  &  -6600   &   46200 \\
ac^{2}&0& 0  &  0   &  0\\
b^{2}c&0& 600  &  0   &  0\\
be&0& 0  &  21600  &  0\\
cd&0&0  &   0  &   108000
\end{matrix} \] 
    \[    \begin{matrix}
    a^{8} &  1 &  7 &   21 &   221  &  1457 &   -2659  &  18341\\
     a^{6}b  &  -19 &  -83  &  -249  &  -2049   &  -14133 &   10671 &  -142329\\
    a^{5}c  &  40  &  -320  &  40 &  -1960   &  -5320  &   191840 &  -813160\\
    a^{4}b^{2} &  50 &  0 &  0 &  0  &  0 &   0 &  0\\
    a^{4}d &  0 &  2500 &  -5500  &  -167500  &   -1049500  &   2835500 &  -8483500\\
     a^{3}bc  &   0 &   2000  &   0 &  0   &  0 &   0 &  0\\
    a^{3}e &  0 &  0 &   60000  &  1680000  &  10320000  &   -32592000 &   122640000 \\
    a^{2}bd &  0 &  0 &  30000  &  0  &  0 &   0 &  0\\
    b^4 &  0 &  0 &  0 &   30000  &  0 &   0 &  0\\
     b^{2}d &  0 &  0 &  0 &  0  &  2700000  &   0 &  0\\
     ce &  0 &  0 &  0 &  0  &  0 &   97200000 &  0\\
    d^{2} & 0  &  0 &  0 &  0  &  0 &   0 &  243000000
    \end{matrix} \] \phantom{X} \\

The relations in norm $9$ and $10$ may be summarized by
\begin{align*}
    a^9 = 10 a^7b = 100 a^6c = 1000 a^5d = 10^7de, \\
    a^{10} = e^2 = 0.
\end{align*}

\end{document}